\documentclass{amsart}
\usepackage{amssymb}
\usepackage{amsmath}
\usepackage{amsthm}
\usepackage{color}

\usepackage{graphicx}
\usepackage{subfigure}
\usepackage{tikz}
\usepackage{array}
\usepackage{multirow}

\usepackage{geometry}

\usepackage[ruled]{algorithm2e}

\usepackage{bbm}
\usepackage{dsfont}

\newtheorem{defn}{Definition}[section]
\newtheorem{thm}{Theorem}[section]
\newtheorem{lem}[thm]{Lemma}
\newtheorem{coro}[thm]{Corollary}
\newtheorem{rmk}[thm]{Remark}
\newtheorem{assum}{Assumption}[section]

\allowdisplaybreaks

\title{The Random Feature Method for Time-dependent Problems}

\author{Jingrun Chen}
\address{School of Mathematical Sciences, University of Science and Technology of China, Hefei 230026, China and Suzhou Institute for Advanced Research, University of Science and Technology of China, Suzhou 215123, China.}
\email {jingrunchen@ustc.edu.cn}

\author{Weinan E}
\address{AI for Science Institute, Beijing and Center for Machine Learning Research and School of Mathematical Sciences, Peking University.}
\email {weinan@math.pku.edu.cn}

\author{Yixin Luo}
\address{University of Science and Technology of China, Hefei 230026, China and Suzhou Institute for Advanced Research, University of Science and Technology of China, Suzhou 215123, China.}
\email[Corresponding author] {seeing@mail.ustc.edu.cn}

\dedicatory{Dedicated to Professor Tao Tang on the occasion of his 60th birthday}

\begin{document}


\begin{abstract}
We present a framework for solving time-dependent partial differential equations (PDEs) in the spirit of the random feature method. The numerical solution is constructed using a space-time partition of unity and random feature functions. Two different ways of constructing the random feature functions are investigated:
feature functions that treat the spatial and temporal variables (STC) on the same footing,  or functions that are the product of two random feature functions depending on spatial and temporal variables separately (SoV). Boundary and initial conditions are enforced by penalty terms. 
We also study two ways of solving the resulting least-squares problem:  the problem is solved as a whole or solved using the block time-marching strategy. 
The former is termed ``the space-time random feature method'' (ST-RFM). Numerical results for a series of  problems  show that the proposed method, i.e. ST-RFM with STC and ST-RFM with SoV, have  spectral accuracy in both space and time. In addition, ST-RFM only requires collocation points, not a mesh.
This is important  for solving problems with complex geometry. We demonstrate this by using ST-RFM to solve a two-dimensional wave equation over a complex domain.  The two strategies differ significantly in terms of the behavior in time. 
In the case when  block time-marching is used, we prove a lower error bound that shows an exponentially growing factor with respect to the number of blocks in time. For ST-RFM, we prove an upper bound with a sublinearly growing factor with respect to the number of subdomains in time. These estimates are also confirmed by  numerical results. 
\end{abstract}

\maketitle


\section{Introduction}

Time-dependent partial differential equations (PDEs), such as diffusion equation, wave equation, Maxwell equation, and Schr\"{o}dinger equation, are widely used for modeling the dynamic evolution of physical systems. Numerical methods, including finite difference method~\cite{leveque2007finite}, finite element methods~\cite{thomee2007galerkin}, and spectral methods~\cite{shen2011spectral}, have been proposed to solve these PDEs. Despite the great success in theory and application, these methods still face some challenges, to name a few, complex geometry, mesh generation, and possibly high dimensionality. 

Along another line, the success of deep learning in computer vision and natural language processing~\cite{goodfellow2016deep} attracts great attention in the community of scientific computing. As a special class of functions, neural networks are proved to be universal approximators to continuous functions~\cite{cybenko1989approximation}. Many researchers seek for solving ordinary and partial differential equations with neural networks~\cite{EHJ2017,HJE2018,weinan2018ritz,justin2018dgm,zang2020wan,raissi2019pinn,weinan2021algorithms}. Since the PDE solution can be defined in the variational (if exists), strong, and weak forms, deep Ritz method~\cite{weinan2018ritz}, deep Galerkin method~\cite{justin2018dgm} and physics-informed neural networks~\cite{raissi2019pinn}, and weak adversarial network~\cite{zang2020wan} are proposed using loss (objective) functions in the variational, strong, and weak forms, respectively. Deep learning-based algorithms have now made it fairly routine to solve a large class of PDEs in high dimensions without the need for mesh generation of any kind.

For low-dimensional problems, traditional methods are accurate, with reliable error control, stability analysis and affordable cost. However, in practice, coming up with a suitable mesh is often a highly non-trivial task, especially for complex geometry. On the contrary, machine-learning methods are mesh-free and only collocation points are needed. Even for low-dimensional problems, this point is still very attractive. What bothers a user is the lackness of reliable error control in machine-learning methods. For example, without an exact solution, the numerical approximation given by a machine-learning method does not show a clear trend of convergence as the number of parameters increases. 

There are some efforts to combine the merits of traditional methods and deep-learning based methods. The key ingredient is to replace deep neural networks by a special class of two-layer neural networks with the inner parameters fixed, known as random features \cite{Neal1995BayesianLF,NIPS2007_013a006f} or extreme learning machine~\cite{huang2006extreme}. Random feature functions are proved to be universal approximators as well, meanwhile only the parameters of the output layer need to be optimized, leading to a convex optimization problem. Extreme learning machines are employed to solve ordinary and partial differential equations in \cite{yang2018novel} and \cite{francesco2021extreme}, respectively.
Spectral accuracy is obtained for problems with analytic solutions, and the simplicity of network architectures reduces the training difficulty in terms of execution time and solution accuracy, compared to deep neural networks. In~\cite{dong2021local}, a special kind of partition of unity (PoU), termed as domain decomposition, is combined with extreme learning machines to approximate the PDE solution and the block time-marching strategy is proposed for long time simulations. Spectral accuracy is obtained in both space and time for analytic solutions, but the error grows exponentially fast in most cases as the simulation time increases. In~\cite{chen2022bridging}, combining PoU and random feature functions, the random feature method (RFM) is proposed to solve static PDEs with complex geometries. An automatic rescaling strategy is also proposed to balance the weights of equations and boundary conditions, which is found to work well for linear elasticity and Stokes flow over complex geometrices. 

The objective in this article is to propose a methodology for solving time-dependent PDEs that shares the merits of both traditional and machine learning-based algorithms. This new class of algorithms can be made spectrally accurate in both space and time. Meanwhile, they are also mesh-free,
making them easy to use even in settings with complex geometry. Our starting point is based on a combination of rather simple and well-known ideas: 
We use space-time PoU and random feature functions to represent the approximate solution, the collocation method to take care of the PDE as well as the boundary conditions in the least-squares sense, and a rescaling procedure to balance the contributions from the PDE and initial/boundary conditions in the loss function. This method is called ST-RFM. For time-dependent problems, a random feature function can depend on both spatial and temporal variables, i.e. space-time concatenation input (STC), or is the product of two random feature functions depending on spatial and temporal variables separately (SoV). STC can be viewed as natural extensions of random feature method for static problems~\cite{chen2022bridging}, while SoV may be a better choice for some time-dependent PDEs. Both STC and SoV are proved to be universal approximators. For long time intervals, PoU for the temporal variable, i.e. domain decomposition along the time direction, is propsed to solve time-dependent problems. Error estimates of the ST-RFM and the block-time marching strategy are provided. ST-RFM yields spectrally accurate results with slowly growing error in terms of the number of subdomains in time, while the error generated by the block time-marching strategy in \cite{dong2021local} grows exponentially fast in terms of the number of blocks. These findings are confirmed by numerical results in one and two dimensions.

This article is organized as follows. In Section~\ref{sec:method}, we present the ST-RFM and prove the approximation property of STC and SoV. Upper bound error estimate for ST-RFM and lower bound error estimate for the block time-marching strategy are provided. In Section \ref{sec:num}, numerical experiments in one and two dimensions for heat, wave, and Schr\"{o}dinger equations are conducted to show the spectral accuracy of the proposed method, and confirm the error estimates for long time simulations. Application of ST-RFM to a two-dimensional wave equation over a complex geometry is also provided. Conclusions are drawn in Section~\ref{sec:con}.

\section{Random Feature Method for Solving Time-dependent PDEs} \label{sec:method}

For completeness, we first recall the RFM for solving static problems. We then introduce the ST-RFM for solving time-dependent PDEs, and prove the universal approximation property and the error estimate.

\subsection{Random Feature Method for Static Problems} \label{sec:method:rfm}

Let $\boldsymbol{x}\in \Omega \subset \mathbb{R}^{d_x}$, where $d_x \in \mathbb{N}^+$ is the dimension of $\boldsymbol{x}$, and let $d_u \in \mathbb{N}^{+}$ be the dimension of output. Consider the following boundary-value problem
\begin{equation}
    \left\{
    \begin{aligned}
        \mathcal{L} \boldsymbol{u} (\boldsymbol{x}) &= \boldsymbol{f}(\boldsymbol{x}), & \quad \boldsymbol{x} \in \Omega, \\
        \mathcal{B} \boldsymbol{u} (\boldsymbol{x}) &= \boldsymbol{g}(\boldsymbol{x}), & \quad \boldsymbol{x} \in \partial \Omega,
    \end{aligned}
    \right.
    \label{eqn:method:pde}
\end{equation} 
where $\boldsymbol{f}$ and $\boldsymbol{g}$ are known functions, $\mathcal{L}$ and $\mathcal{B}$ are differential and boundary operators, respectively. $\partial \Omega$ is the boundary of $\Omega$.

The RFM has the following components. First, $N$ points $\{ \hat{\boldsymbol{x}}_i \}$ are chosen from $\Omega$, typically uniformly distributed. Then, $\Omega$ is decomposed to $N$ subdomains $\{ \Omega_i \}$ with $\hat{\boldsymbol{x}}_i \in \Omega_i$. Thus, we have $\Omega \subset \cup_i \Omega_i$. For each $\Omega_i$, a PoU function $\psi_i$ with support $\Omega_i$, i.e. $\mathrm{supp}(\psi_i) = \Omega_i$, is constructed. In one dimension, two PoU functions are commonly used:
\begin{subequations}
    \begin{align}
        \psi_{a}(x) &= \mathds{I}_{[-1, 1]}(x), \\
        \psi_b(x) &= \mathds{I}_{[-\frac{5}{4}, -\frac{3}{4}]}(x) \frac{1 + \sin(2 \pi x)}{2} + \mathds{I}_{[-\frac{3}{4}, \frac{3}{4}]}(x) + \mathds{I}_{[\frac{3}{4}, \frac{5}{4}]}(x) \frac{1 - \sin(2 \pi x)}{2}.
    \end{align}
\end{subequations}
The first PoU is discontinuous while the second one is continuously differentiable. In high dimensions, the PoU function $\psi$ can be constructed as a tensor product of $d_x$ one-dimensional PoU functions $\psi$, i.e. $\psi({\boldsymbol{x}}) = \Pi_{i=1}^{d_x} \psi(x_i)$.

Then, for each $\Omega_i$, RFM constructs $J_n \in \mathbb{N}^+$ random features $\phi_{ij}$ by a two-layer neural network with random but fixed parameters $\boldsymbol{k}_{ij}$ and $b_{ij}$, i.e. 
\begin{equation}
    \phi_{ij}(\boldsymbol{x}) = \sigma(\boldsymbol{k}_{ij}^\top l_i(\boldsymbol{x}) + b_{ij}), \qquad j=1, 2, \cdots, J_n, \label{eqn:method:phi}
\end{equation}
where the nonlinear activation function $\sigma$ is chosen as tanh or trigonometric functions in~\cite{chen2022bridging}. In~\eqref{eqn:method:phi}, each component of $\boldsymbol{k}_{ij}$ and $b_{ij}$ is independently sampled from $\mathds{U}(-R_m, R_m)$, where $R_m \in \mathbb{R}^+$ controls the magnitude of parameters. In particular, both the weights and the biases are fixed in the optimization procedure. Moreover, $l_i(\boldsymbol{x})$ is a linear transformation to transform inputs in $\Omega_i$ to $[-1, 1]^{d_x}$. The approximate solution in RFM is the linear combination of these random features together with the PoU
\begin{equation}
    u_{M}(\boldsymbol{x}) = \sum_{i=1}^{N} \psi_i(\boldsymbol{x}) \sum_{j=1}^{J_n} u_{ij} \phi_{ij}(\boldsymbol{x}), \label{eqn:method:rfm}
\end{equation}
where $u_{ij} \in \mathbb{R}$ are unknown coefficients to be sought and $M = N \cdot J_n$ denotes the degree of freedom. For vectorial solutions, we approximate each component of the solution by~\eqref{eqn:method:rfm}, i.e.
\begin{equation}
    \boldsymbol{u}_{M}(\boldsymbol{x}) = (u_{M}^1(\boldsymbol{x}), \cdots, u_{M}^{d_u}(\boldsymbol{x}))^\top. \label{eqn:method:rfm-vec}
\end{equation}

To find the optimal set of parameters $\{u_{ij}\}$, RFM evaluates the original problem on collocation points and formulates a least-squares problem. To be specific, RFM samples $Q \in \mathbb{N}^+$ collocation points $\boldsymbol{x}_{i,1}, \cdots, \boldsymbol{x}_{{i,Q}}$ in each $\Omega_i$, and computes the rescaling parameters $\lambda_{i,k,q}^{p} > 0$ and $\lambda_{i,k,e}^{b} > 0$ for $i \in [N]$, $k \in [d_u]$, $q \in [Q]$ and $e \in [Q]$ satisfying $\boldsymbol{x}_{i, e} \in \partial \Omega$. Let $\lambda_{i,q}^{p} = \mathrm{diag}(\lambda_{i,1,q}^{p}, \cdots, \lambda_{i,d_u,q}^{p})^\top$ and $\lambda_{i,e}^{b} = \mathrm{diag}(\lambda_{i,1,e}^{b}, \cdots, \lambda_{i,d_u,e}^{b})^\top$. Then, the random feature method minimizes
\begin{equation}
    \mathrm{Loss}(\{ u_{i,j,k} \}) = \sum_{i=1}^{N} \left( \sum_{q=1}^{Q} \| \lambda^{p}_{i,q} (\mathcal{L} \boldsymbol{u}_{M} (\boldsymbol{x}_{i,q}) - \boldsymbol{f}(\boldsymbol{x}_{i,q}) ) \|_{2}^2 \right) + \left( \sum_{\boldsymbol{x}_{i, e} \in \partial \Omega} \| \lambda^{b}_{i,e} (\mathcal{B} \boldsymbol{u}_{M} (\boldsymbol{x}_{i,e}) - \boldsymbol{g}(\boldsymbol{x}_{i,e})) \|_{2}^2 \right), \label{eqn:method:ls}
\end{equation}
where $\boldsymbol{u}_{M}$ is of the form~\eqref{eqn:method:rfm-vec} and $u_{i,j,k}$ is the $k$-th coefficient of the random feature $\phi_{ij}$.

The above problem~\eqref{eqn:method:ls} is a linear least-squares problem when $\mathcal{L}$ and $\mathcal{B}$ are linear operators. Moreover, when the discontinuous PoU $\psi_a$ is used, continuity conditions between adjacent subdomains must be imposed by adding regularization terms in~\eqref{eqn:method:ls}, while no regularization is required when $\psi_b$ is used for second-order equations. By minimizing~\eqref{eqn:method:ls}, the optimal coefficients $\boldsymbol{u}^* = (u_{ijk}^{*},)^\top$ are obtained and the numerical solution is constructed by~\eqref{eqn:method:rfm-vec}.

\subsection{Space-Time Random Feature Method} \label{sec:method:stpou}

Now, we consider time-dependent PDEs of the following form with the final time $T > 0$
\begin{equation}
    \left\{
    \begin{aligned}
        \mathcal{L} \boldsymbol{u} (\boldsymbol{x}, t) &= \boldsymbol{f} (\boldsymbol{x}, t), & \quad \boldsymbol{x}, t \in \Omega \times [0, T], \\
        \mathcal{B} \boldsymbol{u} (\boldsymbol{x}, t) &= \boldsymbol{g} (\boldsymbol{x}, t), & \quad \boldsymbol{x}, t \in \partial \Omega \times [0, T], \\
        \mathcal{I} \boldsymbol{u} (\boldsymbol{x}, 0) &= \boldsymbol{h} (\boldsymbol{x}), & \quad \boldsymbol{x} \in \Omega, \\
    \end{aligned}
    \right.
    \label{eqn:method:tpde}
\end{equation}
where $\boldsymbol{f}$, $\boldsymbol{g}$ and $\boldsymbol{h}$ are known functions and $\mathcal{I}$ is the initial operator.

Following the same routine as in RFM, we construct a partition of $\Omega \times [0, T]$. First, we decompose the spatial domain $\Omega$ to $N_x$ subdomains where each subdomain $\Omega_i$ contains a central point $\hat{\boldsymbol{x}}_i$. We then decompose the temporal interval $[0, T]$ to $N_t$ subdomains, i.e. $[0, T) = [t_0, t_1) \cup [t_1, t_2), \cdots, [t_{N_t - 1}, t_{N_t})$, where each subdomain contains a central point $\hat{t}_i = \frac{t_{i-1} + t_{i}}{2}$. The product of the PoUs in space and in time results in the space-time PoU, i.e.
\begin{equation}
    \psi_{i_x, i_t}(\boldsymbol{x}, t) = \psi_{i_x}(\boldsymbol{x}) \psi_{i_t}(t),
\end{equation}
where $i_x$ and $i_t$ are indices for spatial and temporal subdomains, respectively.

Next, we generalize spatial random features~\eqref{eqn:method:phi} to space-time random features. There are two options. The first can be viewed as a natural extension of \eqref{eqn:method:phi}, where the concatenation of spatial and temporal variables is fed into $\phi_{ij}$ as the input and the output is a random feature function, i.e.
\begin{equation}
    \phi_{i_x, i_t, j}(\boldsymbol{x}, t) = \sigma ((\boldsymbol{k}_{i_x, i_t, j}^{x})^\top l_{i_x}(\boldsymbol{x}) + k^{t}_{i_x, i_t, j} l_{i_t}(t) + b_{i_x, i_t, j}). \label{eqn:method:concat}
\end{equation}
Here $\boldsymbol{k}_{i_x, i_t, j}^{x}$ and $k_{i_x, i_t, j}^{t}$ are the weights associated with spatial and temporal inputs, respectively. $b_{i_x, i_t, j}$ is the bias, $l_{i_x}(\boldsymbol{x})$ and $l_{i_t}(t)$ are linear transformations from $\boldsymbol{x} \in \Omega_{i_x} $ to $[-1, 1]^{d_x}$ and from $t \in [t_{i_t - 1}, t_{i_t}]$ to $[-1, 1]$, respectively. \eqref{eqn:method:concat} is called STC.

The second option is to use separation of variables, which mimics the technique of separation of variables for solving PDEs
\begin{equation}
    \phi_{i_x, i_t, j}(\boldsymbol{x}, t) = \sigma ((\boldsymbol{k}_{i_x, i_t, j}^{x})^\top l_{i_x}(\boldsymbol{x}) + b^{x}_{i_x, i_t, j}) \sigma (k^{t}_{i_x, i_t, j} l_{i_t}(t) + b_{i_x, i_t, j}^{t}). \label{eqn:method:sov}
\end{equation}
In this formulation, the space-time random feature is a product of the spatial random feature and the temporal random feature, thus we term it as SoV. For both random features, the degrees of freedom $M = N_x \cdot N_t \cdot J_n$. 

The combination of these space-time random features and PoU leads to the following approximate solution in the scalar case
\begin{equation}
    u_{M}(\boldsymbol{x}, t) = \sum_{i_x=1}^{N_x} \sum_{i_t=1}^{N_t} \psi_{i_x}(\boldsymbol{x}) \psi_{i_t}(t) \sum_{j=1}^{J_n} u_{i_x,i_t,j} \phi_{i_x,i_t,j}(\boldsymbol{x}, t), \label{eqn:method:tnsol}
\end{equation}
where $u_{i_x, i_t, j}$ is the unknown coefficient to be sought. A vectorial solution $\boldsymbol{u}_{M}$ can be formulated in the same way as that in~\eqref{eqn:method:rfm-vec}, i.e.
\begin{equation}
    \boldsymbol{u}_{M}(\boldsymbol{x}, t) = (u_{M}^{1}(\boldsymbol{x}, t), \cdots, u_{M}^{d_u}(\boldsymbol{x}, t))^\top. \label{eqn:method:tnsol-vec}
\end{equation}
For the $k$-th component of $\boldsymbol{u}_{M}$, we denote the coefficient associated to $\phi_{i_x,i_t,j}$ by $u_{i_x, i_t, j, k}$.

By substituting \eqref{eqn:method:tnsol} or \eqref{eqn:method:tnsol-vec} into \eqref{eqn:method:tpde}, we define the loss function on collocation points. Denote $\{\boldsymbol{x}_{i_x, i_t, 1}, \cdots, \boldsymbol{x}_{i_x, i_t, Q_x}\}$ spatial collocation points in the $i_x$-th spatial subdomain and $(i_t - 1) T/N_t = t_{i_x, i_t} = t_{i_x, i_t, 0} < t_{i_x, i_t, 1} < \cdots < t_{i_x, i_t, Q_t} = i_t T/N_t$ temporal collocation points in the $i_t$-th temporal subdomain, respectively. In RFM, we define the rescaling parameters $\lambda_{i_x, i_t, q_x, q_t}^{e}$, $\lambda_{i_x, i_t, e, q_t}^{b}$ and $\lambda_{i_x, i_t, q_x}^{i}$ for $i_x \in [N_x]$, $i_t \in [N_t]$, $q_x \in [Q_x]$, $q_t \in [Q_t]$ and $e \in [Q_x]$ satisfying $\boldsymbol{x}_{i_x, i_t, e} \in \partial \Omega$. Then, the ST-RFM minimizes 
\begin{equation}
    \begin{aligned}
        \mathrm{Loss}(\{ u_{i_x, i_t, j, k} \}) = \sum_{i_x=1}^{N_x} \sum_{i_t=1}^{N_t} & \left( \sum_{q_x=1}^{Q_x} \sum_{q_t=1}^{Q_t} \| \lambda^{e}_{i_x, i_t, q_x, q_t} (\mathcal{L} \boldsymbol{u} (\boldsymbol{x}_{i_x, i_t, q_x}, t_{i_x, i_t, q_t}) - \boldsymbol{f}(\boldsymbol{x}_{i_x, i_t, q_x}, t_{i_x, i_t, q_t})) \|^2 \right) + \\
        &\left( \sum_{\boldsymbol{x}_{i_x, i_t, e} \in \partial \Omega} \sum_{q_t=1}^{Q_t} \| \lambda^{b}_{i_x, i_t, e, q_t} (\mathcal{B} \boldsymbol{u} (\boldsymbol{x}_{i_x, i_t, e}, t_{i_x, i_t, q_t}) - \boldsymbol{g}(\boldsymbol{x}_{i_x, i_t, e}, t_{i_x, i_t, q_t})) \|^2 \right) + \\
        & \mathds{I}_{i_t = 1} \left( \sum_{q_x=1}^{Q_x} \| \lambda^{i}_{i_x, i_t, q_x} (\mathcal{I} \boldsymbol{u} (\boldsymbol{x}_{i_x, i_t, q_x}, 0) - \boldsymbol{h}(\boldsymbol{x}_{i_x, i_t, q_x})) \|^2 \right)
    \end{aligned}
    \label{eqn:method:ls-st}
\end{equation}
to find the numerical solution of the form~\eqref{eqn:method:tnsol-vec}. Moreover, when $N_x > 1$ or $N_t > 1$ and $\psi_a$ is used, continuity conditions between adjacent spatial (temporal) subdomains must be imposed by adding regularization terms in~\eqref{eqn:method:ls-st}, while no regularization is required when $\psi_b$ is used for second-order equations.

For better illustration of this method, we consider an example with $d_u = 1$, the initial operator $\mathcal{I}$ being identity and rescaling parameters being one. For convenience, we set $N_x = 1$ and relabel $i_t$ by $i$ without confusion. Under these settings, we introduce two matrix-valued functions $\mathbf{\Phi}_{i}(t)$, $\mathbf{L}_i(t)$ and three induced matrices as follows:
\begin{align*}
    \mathbf{\Phi}_{i}(t) =& (\phi_{i, j}(\boldsymbol{x}_{i, q}, t),) \in \mathbb{R}^{Q_x \times J_n}, \quad \mathbf{\Phi}_{i,0} = \mathbf{\Phi}_{i}(t_{i, 0}), \quad \mathbf{\Phi}_{i,1} = \mathbf{\Phi}_{i}(t_{i, Q_t}), \\
    \mathbf{L}_{i}(t) =& (\mathcal{L} \phi_{i, j} (\boldsymbol{x}_{i, q}, t),) \in \mathbb{R}^{Q_x \times J_n}, \quad \mathbf{L}_{i} = [\mathbf{L}_{i}(t_{i, 0})^\top, \cdots, \mathbf{L}_{i}(t_{i, Q_t - 1})^\top]^{\top}, 
\end{align*}
where $q$ is the row index and $j$ is the column index.
Then, we construct the matrix $\mathbf{A}$ as follows:
\begin{equation*}
    \mathbf{A}_{i} = [\mathbf{0}_{(J_n \times (i-1) (Q_t + 1)Q_x)}, \mathbf{\Phi}^\top_{i, 0}, \mathbf{L}_{i}^\top, -\mathbf{\Phi}_{i, 1}^\top, \mathbf{0}_{(J_n \times ((N_t - i)(Q_t + 1) - 1)Q_x)}]^\top, \quad \mathbf{A} = [\mathbf{A}_{1}, \cdots, \mathbf{A}_{N_t}].
\end{equation*}
Define the following vectors
\begin{equation*}
    \begin{aligned}
        & \boldsymbol{h} = [h(\boldsymbol{x}_{1, 1}), h(\boldsymbol{x}_{1, 2}), \cdots, h(\boldsymbol{x}_{1, Q_x})]^\top, \\ 
        & \boldsymbol{f}_{i,j} = [f(\boldsymbol{x}_{i, 1}, t_{i, j}), f(\boldsymbol{x}_{i, 2}, t_{i, j}), \cdots, f(\boldsymbol{x}_{i, Q_x}, t_{i, j})]^\top, \\ 
        & \boldsymbol{f}_{i} = [\boldsymbol{f}_{i, 0}^{\top}, \boldsymbol{f}_{i, 1}^{\top}, \cdots, \boldsymbol{f}_{i, Q_t-1}^{\top}]^\top,
    \end{aligned}
\end{equation*}
then we construct the vector $\boldsymbol{b} = \sum_{i=1}^{N_t} \boldsymbol{b}_i$, where
\begin{equation*}
    \begin{aligned}
        & \boldsymbol{b}_1 = [\boldsymbol{h}^\top, \boldsymbol{f}_1^\top, \boldsymbol{0}_{((N_t-1)(Q_t+1)Q_x)}^\top]^\top, \\
        & \boldsymbol{b}_i = [\boldsymbol{0}_{(((i-1)(Q_t+1)+1)Q_x)}^\top, \boldsymbol{f}_i^\top, \boldsymbol{0}_{((N_t-i)(Q_t+1)Q_x)}^\top], \quad \mathrm{for} \; i=2, \cdots, N_t.
    \end{aligned}
\end{equation*}
Let $\boldsymbol{u} \in \mathbb{R}^{N_t J_n}$, then the optimal coefficient $\boldsymbol{u}^{S}$ by ST-RFM is obtained by
\begin{equation}
    \boldsymbol{u}^{S} = \min_{\boldsymbol{u}} \| (\mathbf{A}\boldsymbol{u} - \boldsymbol{b}) \|^2.
    \label{eqn:method:strfm}
\end{equation}

For time-dependent partial differential equations, Dong et al.~\cite{dong2021local} proposed the block time-marching strategy. Block time-marching strategy solves Eq.~\eqref{eqn:method:tpde} in each time block individually and applies the numerical results in $i$-th block at the terminal time as the initial conditions of $(i+1)$-th block. Let $N_b \in \mathbb{N}^+$ be the number of time blocks, and let $N_t = 1$ for simplicity. For first-order equations in time, we present the detailed process of block time-marching strategy in Algorithm~\ref{alg:method:timeblock}, and the optimal coefficient $\boldsymbol{u}^{B}$ is the concatenation of optimal coefficients in all time blocks.
\begin{algorithm}[t]
    \SetAlgoLined
    \LinesNumbered
    \KwOut{$\boldsymbol{u}^{B}$}
    $\tilde{\mathbf{A}}_{1} = [\boldsymbol{\Phi}_{1, 0}^\top, \boldsymbol{L}_{1}^\top ]^\top$, $\tilde{\boldsymbol{b}}_{1} = [\boldsymbol{g}^\top, \boldsymbol{f}_{1}^\top]^\top$ \;
    $\boldsymbol{u}_{1}^{B} = \min_{\boldsymbol{u}} \| \tilde{\mathbf{A}}_{1} \boldsymbol{u} - \tilde{\boldsymbol{b}}_{1} \|^2$ \;
    \For{i=2,3,$\cdots$,$N_b$}{
        $\tilde{\mathbf{A}}_{i} = [\boldsymbol{\Phi}_{i, 0}^\top, \boldsymbol{L}_{i}^\top ]^\top$, $\tilde{\boldsymbol{b}}_{i} = [ (\boldsymbol{u}_{i-1}^{B})^{\top} \boldsymbol{\Phi}_{i-1, 1}^\top, \boldsymbol{f}_{i}^\top]^\top$ \;
        $\boldsymbol{u}_{i}^{B} = \min_{\boldsymbol{u}} \| \tilde{\mathbf{A}}_{i} \boldsymbol{u} - \tilde{\boldsymbol{b}}_{i} \|^2$ \;
    }
    $\boldsymbol{u}^{B} = [(\boldsymbol{u}_{1}^{B})^\top, (\boldsymbol{u}_{2}^{B})^\top, \cdots, (\boldsymbol{u}_{N_b}^{B})^\top]^\top$ \;
    \caption{Block time-marching strategy to solve time-dependent PDEs}
    \label{alg:method:timeblock}
\end{algorithm}

\subsection{Approximation Properties} \label{sec:method:theory1}

Both concatenation random features and separation-of-variables random features are universal approximators. To prove these, we first recall the definition of sigmoidal functions in~\cite{cybenko1989approximation}.
\begin{defn}
    $\sigma$ is said to be sigmoidal if
    \begin{equation*}
        \sigma(z) \to \begin{cases}
            1 \quad &\mathrm{as} \; z \to +\infty, \\
            0 \quad &\mathrm{as} \; z \to -\infty. \\
        \end{cases}
    \end{equation*}
\end{defn}

The approximation property of STC is given in the following theorem.
\begin{thm} \label{thm:method:1}
    Let $\sigma$ be any continuous sigmoidal function. Then finite sums of
    \begin{equation}
        G_C(\boldsymbol{x}, t) = \sum_{j=1}^N u_j \sigma((\boldsymbol{k}_{j}^{x})^\top \boldsymbol{x} + k_{j}^{t} t + b_{j}) \label{eqn:method:fsum}
    \end{equation}
    are dense in $C(\Omega \times [0, T])$. In other words, given any $f \in C(\Omega \times [0, T])$ and $\epsilon > 0$, there exists a sum $G_C(x, t)$, such that
    \begin{equation*}
        |G_C(\boldsymbol{x}, t) - f(\boldsymbol{x}, t)| < \epsilon \quad \mathrm{for}\; \mathrm{all} \; \boldsymbol{x}, t \in \Omega \times [0, T].
    \end{equation*}
\end{thm}

\begin{proof}
    It is a direct consequence of  in~\cite[Theorem~2]{cybenko1989approximation}.
\end{proof}

\begin{coro} \label{coro:method:1}
    When using tanh as the activation function in~\eqref{eqn:method:fsum}, the finite sums~\eqref{eqn:method:fsum} are dense in $C(\Omega \times [0, T])$.
\end{coro}

\begin{proof}
    Consider a continuous sigmoidal function $\tilde{\sigma}(z) = \frac{1}{1 + e^{-z}}$, then $\mathrm{tanh}(z) = 2 \tilde{\sigma}(2 z) - 1$. From Theorem~\ref{thm:method:1}, there is a sum $\tilde{G}_J(\boldsymbol{x}, t) = \sum_{j=1}^{\tilde{N}} \tilde{u}_j \tilde{\sigma}((\boldsymbol{\tilde{k}}_{j}^{x})^\top \boldsymbol{x} + \tilde{k}_{j}^{t} t + \tilde{b_j})$, such that
    \begin{equation*}
        |\tilde{G}_J(\boldsymbol{x}, t) - f(\boldsymbol{x}, t)| < \frac{\epsilon}{2}.
    \end{equation*}
    Define $\hat{u} = \frac{1}{2} \sum_{j=1}^{\tilde{N}} \tilde{u}_j$. If $\hat{u} = 0$, then the finite sum $ G_C(\boldsymbol{x}, t) = \sum_{j=1}^{\tilde{N}} \frac{1}{2} \tilde{u}_j \mathrm{tanh}(\frac{1}{2} ((\boldsymbol{\tilde{k}}_{j}^{x})^\top \boldsymbol{x} + \tilde{k}_{j}^{t} t + \tilde{b_j})) $ satisfies
    \begin{align*}
        |G_C(\boldsymbol{x}, t) - f(\boldsymbol{x}, t)| &= |\sum_{j=1}^{\tilde{N}} \frac{1}{2} \tilde{u}_j (2 \tilde{\sigma}((\boldsymbol{\tilde{k}}_{j}^{x})^\top \boldsymbol{x} + \tilde{k}_{j}^{t} t + \tilde{b_j}) - 1) - f(\boldsymbol{x}, t)| = |\tilde{G}_J(\boldsymbol{x}, t) - f(x, t)| < \frac{\epsilon}{2} < \epsilon.
    \end{align*}

    If $\hat{u} \neq 0$, since $\Omega \subset \mathbb{R}^{d_x}$ is compact and thus is bounded and closed, there exist $\boldsymbol{\hat{k}}^{x}$, $\hat{k}^{t}$ and $\hat{b}$ such that
    \begin{equation*}
        |\mathrm{tanh}((\boldsymbol{\hat{k}}^{x})^\top \boldsymbol{x} + \hat{k}_{t} t + \hat{b}) - 1| < \frac{\epsilon}{|\hat{u}|}
    \end{equation*}
    holds for all $(\boldsymbol{x}, t) \in \Omega \times [0, T]$.
    Then, the finite sum
    $ G_C(\boldsymbol{x}, t) = \sum_{j=1}^{\tilde{N}} \frac{1}{2} \tilde{u}_j \mathrm{tanh}(\frac{1}{2}((\boldsymbol{\tilde{k}}_{j}^{x})^\top \boldsymbol{x} + \tilde{k}_{j}^{t} t + \tilde{b_j})) + \frac{1}{2} \hat{u} \mathrm{tanh}((\boldsymbol{\hat{k}}^{x})^\top \boldsymbol{x} + \hat{k}^{t} t + \hat{b}) $ satisfies
    \begin{align*}
        |G_C(\boldsymbol{x}, t) - f(\boldsymbol{x}, t)| &= | \sum_{j=1}^{\tilde{N}} \frac{1}{2} \tilde{u}_j (2 \tilde{\sigma}((\boldsymbol{\tilde{k}}_{j}^{x})^\top \boldsymbol{x} + \tilde{k}_{j}^{t} t + \tilde{b_j}) - 1) + \frac{1}{2} \hat{u} \mathrm{tanh}((\boldsymbol{\hat{k}}^{x})^\top \boldsymbol{x} + \hat{k}^{t} t + \hat{b}) - f(\boldsymbol{x}, t) | \\
        & \le | \tilde{G}_J(\boldsymbol{x}, t) - f(\boldsymbol{x}, t) | + \frac{1}{2} |\hat{u}| |\mathrm{tanh}((\boldsymbol{\hat{k}}^{x})^\top \boldsymbol{x} + \hat{k}^{t} t + \hat{b}) - 1| \\
        & < \frac{\epsilon}{2} + \frac{\epsilon}{2} = \epsilon
    \end{align*}
    for all $(\boldsymbol{x}, t) \in \Omega \times [0, T]$.
\end{proof}

For SoV, we first extend the definition of discriminatory functions in~\cite{cybenko1989approximation} as follows.
\begin{defn}
    $\sigma$ is said to be discriminatory if for a measure $\mu \in \mathcal{M}(\Omega \times [0, T])$
    \begin{equation*}
        \int_{\Omega \times [0, T]} \sigma((\boldsymbol{k}^{x})^\top \boldsymbol{x} + b^{x}) \sigma(k^{t} t + b^{t}) \mathrm{d} \mu(\boldsymbol{x}, t)
    \end{equation*}
    for all $\boldsymbol{k}^{x} \in \mathbb{R}^{d_x}$, $k^{t} \in \mathbb{R}$,  $b^{x} \in \mathbb{R}$ and $b^{t} \in \mathbb{R}$ implies $\mu = 0$.
\end{defn}

\begin{lem} \label{lem:method:1}
    Any bounded, measurable sigmoidal function $\sigma$ is discriminatory. In particular, any continuous sigmoidal function is discriminatory.
\end{lem}
\begin{proof}
    For any $\boldsymbol{x}, t, \boldsymbol{k}=((\boldsymbol{k}^{x})^\top, k_t)^\top, b, \theta=(\theta^{x}, \theta^{t})$, we have
    \begin{equation*}
        \sigma(\lambda((\boldsymbol{k}^{x})^\top \boldsymbol{x} + b) + \theta_x) \begin{cases}
            \to 1 & \mathrm{for} \quad (\boldsymbol{k}^{x})^\top \boldsymbol{x} + b > 0 \quad \mathrm{as} \quad \lambda \to +\infty, \\
            \to 0 & \mathrm{for} \quad (\boldsymbol{k}^{x})^\top \boldsymbol{x} + b < 0 \quad \mathrm{as} \quad \lambda \to +\infty, \\
            = \sigma(\theta_x) & \mathrm{for} \quad (\boldsymbol{k}^{x})^\top \boldsymbol{x} + b = 0 \quad \mathrm{for}\; \mathrm{all} \; \lambda.
        \end{cases}
    \end{equation*}
    Thus, $\sigma_{x, \lambda}(\boldsymbol{x}) = \sigma(\lambda((\boldsymbol{k}^{x})^\top \boldsymbol{x} + b) + \theta_x)$ converges in the pointwise and bounded sense to 
    \begin{equation*}
        \gamma_x(\boldsymbol{x}) = \begin{cases}
            1 & \mathrm{for} \quad (\boldsymbol{k}^{x})^\top \boldsymbol{x} + b > 0, \\
            0 & \mathrm{for} \quad (\boldsymbol{k}^{x})^\top \boldsymbol{x} + b < 0, \\
            \sigma(\theta_x) & \mathrm{for} \quad (\boldsymbol{k}^{x})^\top \boldsymbol{x} + b = 0
        \end{cases}
    \end{equation*}
    as $\lambda \to +\infty$. Similarly, we have $\sigma_{t, \lambda}(t) = \sigma(\lambda(y_t t + \theta) + \phi_t)$ converges in the pointwise and bounded sense to
    \begin{equation*}
        \gamma_t(t) = \begin{cases}
            1 & \mathrm{for} \quad k^{t} t + b > 0, \\
            0 & \mathrm{for} \quad k^{t} t + b < 0, \\
            \sigma(\theta_t) & \mathrm{for} \quad k^{t} t + b = 0
        \end{cases}
    \end{equation*}
    as $\lambda \to +\infty$.

    Denote $\Pi_{k, b}^{x} = \{ (\boldsymbol{x}, t) | (\boldsymbol{k}^{x})^\top \boldsymbol{x} + b = 0, k^{t} t + b > 0 \}$ and $\Pi_{k, b}^{t} = \{ (\boldsymbol{x}, t) | (\boldsymbol{k}^{x})^\top \boldsymbol{x} + b > 0, k^{t} t + b = 0 \}$. Let $I_{k, b}$ be the hyperline defined by $\{ (\boldsymbol{x}, t) | (\boldsymbol{k}^{x})^\top \boldsymbol{x} + b = 0, k^{t} t + b = 0 \}$ and $H_{k, b}$ be the space defined by $\{ (\boldsymbol{x}, t) | (\boldsymbol{k}^{x})^\top \boldsymbol{x} + b > 0, k^{t} t + b > 0 \}$. Then by the Lesbegue Bounded Convergence Theorem, we have
    \begin{align*}
        0 &= \int_{\Omega \times [0, T]} \sigma_{x, \lambda}(\boldsymbol{x}) \sigma_{t, \lambda}(t) \mathrm{d} \mu(\boldsymbol{x}, t) 
        = \int_{\Omega \times [0, T]} \gamma_x(\boldsymbol{x}) \gamma_t(t) \mathrm{d} \mu(\boldsymbol{x}, t) \\
        &= \mu(H_{k, b}) + \sigma(\theta_x) \mu(\Pi_{k, b}^{x}) + \sigma(\theta_t) \mu(\Pi_{k, b}^{t}) + \sigma(\theta_x) \sigma(\theta_t) \mu(I_{k, b})
    \end{align*}
    for all $\theta, b, \boldsymbol{k}$.

    We now show that the measure of all quarter spaces being zero implies that the measure $\mu$ itself must be zero. This would be trivial if $\mu$ was a positive measure but here it is not.

    For a fixed $\boldsymbol{k}$ and a bounded measurable function $h$, we define a linear functional $F$ as
    \begin{equation*}
        F(h) = \int_{\Omega \times [0, T]} h((\boldsymbol{k}^{x})^\top \boldsymbol{x}) h(k^{t} t) \mathrm{d} \mu(\boldsymbol{x}, t).
    \end{equation*}
    Note that $F$ is a bounded functional on $L^{\infty}(\mathbb{R})$ since $\mu$ is a finite signed measure. Set $h$ the indicator function of the interval $[\theta, \infty)$, i.e. $h(u) = 1$ if $u \ge 0$ and $h(u) = 0$ if $u < \theta$, then
    \begin{equation*}
        F(h) = \int_{\Omega \times [0, T]} h((\boldsymbol{k}^{x})^\top \boldsymbol{x}) h(k^{t} t) \mathrm{d} \mu(\boldsymbol{x}, t) = \mu(I_{k, -b}) + \mu(\Pi_{k, -b}^{x}) + \mu(\Pi_{k, -b}^{t}) + \mu(H_{k, -b}) = 0.
    \end{equation*}
    Similarly, $f(h) = 0$ if $h$ is the indicator function of the open interval $(\theta, \infty)$. By linearity, $F(h) = 0$ holds for the indicator function of any interval and for any simple function (sum of indicator functions of intervals). Since simple functions are dense in $L^{\infty}(\mathbb{R})$, $F = 0$.

    In particular, a substitution of bounded measurable functions $s(\boldsymbol{x}, t) = \sin (\boldsymbol{m}_x^\top \boldsymbol{x} + m_t t)$ and $c(\boldsymbol{x}, t) = \cos (\boldsymbol{m}_x^\top \boldsymbol{x} + m_t t)$ gives 
    \begin{equation*}
        F(c + is) = \int_{\Omega \times [0, T]} \cos (\boldsymbol{m}_x^\top \boldsymbol{x} + m_t t) + i \sin (\boldsymbol{m}_x^\top \boldsymbol{x} + m_t t) \mathrm{d} \mu(x, t) = \int_{\Omega \times [0, T]} \exp(i (\boldsymbol{m}_x^\top \boldsymbol{x} + m_t t)) \mathrm{d} \mu(\boldsymbol{m}, t) = 0
    \end{equation*}
    for all $\boldsymbol{m} = (\boldsymbol{m}_x, m_t)$. The Fourier transform of $\mu$ is zero and we have $\mu=0$. Hence, $\sigma$ is discriminatory.
\end{proof}

\begin{rmk}
    Following the same routine, we can prove the approximation property of SoV. For any continuous sigmoidal function $\sigma$, the finite sums
	\begin{equation}
		G_S(\boldsymbol{x}, t) = \sum_{j=1}^N u_j \sigma((\boldsymbol{k}_{j}^{x})^\top \boldsymbol{x} + b_{j}^{x}) \sigma(k^{t} t + b_{j}^{t}) \label{eqn:method:sep-sig}
	\end{equation}
	are dense in $C(\Omega \times [0, T])$.
\end{rmk}

\begin{coro} \label{coro:method:2}
For the tanh activation function in~\eqref{eqn:method:sep-sig}, the finite sums of the form
    \begin{equation}
        G_S(\boldsymbol{x}, t) = \sum_{j=1}^N u_j \mathrm{tanh}((\boldsymbol{k}_{j}^{x})^\top \boldsymbol{x} + b_{j}^{x}) \mathrm{tanh}(k_{j}^{t} t + b_{j}^{t}) \label{eqn:method:sep-tanh}
    \end{equation}
    are dense in $C(\Omega \times [0, T])$.
\end{coro}

\begin{proof}
    The proof is similar to that of Corollary~\ref{coro:method:1}.
\end{proof}

\subsection{Error Estimates of Space-time Random Feature Method} \label{sec:method:theory2}

For convenience, we set $N_b = N_t$ to analyze the error of block time-marching strategy and the ST-RFM. Let $\boldsymbol{u}^{S}$ and $\boldsymbol{u}^{B}$ be exact solutions in Eq.~\eqref{eqn:method:strfm} and the block time-marching strategy in Algorithm~\ref{alg:method:timeblock}, and $\hat{\boldsymbol{u}}^{S}$ and $\hat{\boldsymbol{u}}^{B}$ be numerical solutions by a numerical solver, such as direct or iterative methods, respectively. Then we have the following results.
\begin{thm} \label{thm:method:2}
    Assume that for any linear least-squares problem $\min_{\boldsymbol{u}} \| \mathbf{A} \boldsymbol{u} - \boldsymbol{b} \|^2$ in ST-RFM and block time-marching strategy, there exists $\boldsymbol{u}^*$ such that $\mathbf{A} \boldsymbol{u}^* = \boldsymbol{b}$. Then $\boldsymbol{u}^{B}=\boldsymbol{u}^{S}$.
\end{thm}

\begin{proof}
    Under the assumption, for each subproblem in Algorithm~\ref{alg:method:timeblock}, the optimal solution $\boldsymbol{u}_{i}^{B}$ satisifies
    \begin{equation*}
        \tilde{\mathbf{A}}_i \boldsymbol{u}_i^{B} = \tilde{\boldsymbol{b}}_i, \quad \mathrm{for} \; i=1, 2, \cdots, N_t.
    \end{equation*}
    Recalling the definition of $\mathbf{A}_i$ in ST-RFM, we have
    \begin{align*}
        \mathbf{A}_i \boldsymbol{u}_i^{B} =& [\boldsymbol{0}_{((i-1)(Q_t+1)Q_x)}^\top, \tilde{\boldsymbol{b}}_i^\top, - (\boldsymbol{u}_i^{B})^{\top} \mathbf{\Phi}_{i,1}^\top, \mathbf{0}_{(((N_t - i)(Q_t + 1) - 1)Q_x)}^\top ]^\top \\
        =& [\boldsymbol{0}_{((i-1)(Q_t+1)Q_x)}^\top, (\boldsymbol{u}_{i-1}^{B})^{\top} \mathbf{\Phi}_{i - 1,1}^\top, \boldsymbol{f}_i^\top, - (\boldsymbol{u}_i^{B})^{\top} \mathbf{\Phi}_{i,1}^\top, \mathbf{0}_{(((N_t - i)(Q_t + 1) - 1)Q_x)}^\top ]^\top \\
        =& [\boldsymbol{0}_{((i-1)(Q_t+1)Q_x)}^\top, (\boldsymbol{u}_{i-1}^{B})^{\top} \mathbf{\Phi}_{i - 1,1}^\top, \mathbf{0}_{(((N_t - i + 1)(Q_t + 1) - 1)Q_x)}^\top ]^\top + \\
        & [\boldsymbol{0}_{(((i-1)(Q_t+1) + 1)Q_x)}^\top, \boldsymbol{f}_{i}^{\top}, \mathbf{0}_{((N_t - i)(Q_t + 1)Q_x)}^\top ]^\top - \\
        & [\boldsymbol{0}_{(i(Q_t+1)Q_x)}^\top, (\boldsymbol{u}_{i}^{B})^{\top} \mathbf{\Phi}_{i,1}^\top, \mathbf{0}_{(((N_t - i)(Q_t + 1) - 1)Q_x)}^\top ]^\top.
    \end{align*}
    After some algebraic calculations, we have
    \begin{align*}
        \mathbf{A} \boldsymbol{u}^{B} = \sum_{i=1}^{N_t} \mathbf{A}_i \boldsymbol{u}_i^{B} = \boldsymbol{b}.
    \end{align*}
    Since the optimal solution of linear least-squares problem is unique under the assumption, we have $\boldsymbol{u}^{B} = \boldsymbol{u}^{S}$.
\end{proof}

From Theorem~\ref{thm:method:2}, solving the time-dependent PDEs by the block time-marching strategy is equivalent to solving the same problem by the ST-RFM. In practice, however, the numerical solution $\hat{\boldsymbol{u}}$ is different from the optimal solution $\boldsymbol{u}$, and we denote the difference by $\delta u$. For long time intervals, when $N_b$ and $N_t$ increases by one, the solution error at $N_b+1$ block or $N_t$ subdomain is expected to be greater than previous blocks or subdomains due to error accumulation. To quantitatively analyze the error propagation in terms of $N_t$ and $N_b$, we first introduce a matrix $\mathbf{B}_i$ as
\begin{equation*}
    \mathbf{B}_i = (\tilde{\mathbf{A}}_i^{\top} \tilde{\mathbf{A}}_i)^{-1} \mathbf{\Phi}_{i,0}^{\top} \mathbf{\Phi}_{i-1,1} \in \mathbb{R}^{J_n \times J_n}, \qquad \mathrm{for}\; i=1, \cdots, N_t.
\end{equation*}
For simplicity, random feature functions are set to be the same over different time subdomains, i.e. $\phi_{i_1, j} = \phi_{i_2, j}$ for all $1 \le i_1 \neq i_2 \le N_t$. Then, all $\mathbf{B}_i$ are the same and can be denoted by $\mathbf{B}$. To proceed, we need the following assumption for $\mathbf{B}$.
\begin{assum} \label{assum:method:2}
$\mathbf{B}$ is diagonalizable.
\end{assum}
\begin{rmk}
Although this assumption cannot be proved, we find that numerically $\mathbf{B}$ is diagonalizable for all the numerical examples we have tested. This is not a bad assumption since it provides a lower bound error estimate of the block time-marching strategy and is verified numerically as well.
\end{rmk}
Under Assumption~\ref{assum:method:2}, there exists $\{ (\lambda_1, \boldsymbol{b}_1), \cdots, (\lambda_{J_n}, \boldsymbol{b}_{J_n}) \}$ such that $\boldsymbol{b}_{k}$ is the eigenvector of $\mathbf{B}$ with eigenvalue $\lambda_k$ and $\{ \boldsymbol{b}_k \}$ is a unitary orthogonal basis in $\mathbb{R}^{J_n}$. Denote the eigenvalue with the largest modulus by $\lambda_m$, i.e. $\lambda_m = \max(|\lambda_k|)$. Since $\{\boldsymbol{b}_k\}$ forms a basis, there exists $\{ \delta u_{k} \}_{k=1}^{J_n}$ such that $\delta u = \sum_{k=1}^{J_n} \delta u_{k} \boldsymbol{b}_k$, where $\delta u_k$, $k=1, \cdots, J_n$ are independent and identically distributed random variables with $\mathbb{E}[\delta u_k] = 0$, $\mathbb{E}[|\delta u_k|] = \mu > 0$ and $\mathbb{E}[(\delta u_k)^2] = \delta^2 > 0$. 

We need the following lemma to bound a geometric series from above and from below.
\begin{lem} \label{lem:method:2}
    For any $\lambda \in \mathbb{C}$ with $| \lambda | \neq 1$ and $\epsilon \in (0, 1)$, there exists $N \in \mathbb{N}^+$ such that
    \begin{equation*}
        (1 + \epsilon) \frac{C_{\lambda}^n}{| 1 - \lambda |} \ge |\sum_{i=0}^{n-1} \lambda^i| \ge (1 - \epsilon) \frac{C_{\lambda}^n}{| 1 - \lambda |}
    \end{equation*}
    holds for $n > N$, where $C_{\lambda} = \max(1, |\lambda|)$.
\end{lem}

\begin{proof}
    When $| \lambda | < 1$, we have
    \begin{equation*}
        \lim_{n \to \infty} \left| \sum_{i=0}^{n-1} \lambda^i \right| | 1 - \lambda | = \lim_{n \to \infty} \left| \frac{1 - \lambda^n}{1 - \lambda} \right| | 1 - \lambda | = \lim_{n \to \infty} |1 - \lambda^n| = 1.
    \end{equation*}
    Therefore, there exists $N_1 \in \mathbb{N}^+$ such that
    \begin{equation*}
        1 + \epsilon \ge \left| \sum_{i=0}^{n-1} \lambda^i \right| | 1 - \lambda | \ge 1 - \epsilon.
    \end{equation*}
    holds for $n > N_1$.
    Similarly, when $|\lambda| > 1$, we have
    \begin{equation*}
        \lim_{n \to \infty} \left| \sum_{i=0}^{n-1} \lambda^i \right| \frac{| 1 - \lambda |}{|\lambda|^n} = \lim_{n \to \infty} \left| \frac{1 - \lambda^n}{1 - \lambda} \right| \frac{| 1 - \lambda |}{|\lambda|^n} = \lim_{n \to \infty} |\lambda^{-n} - 1| = 1.
    \end{equation*}
    Therefore, there exists $N_2 \in \mathbb{N}^+$ such that
    \begin{equation*}
        1 + \epsilon \ge \left| \sum_{i=0}^{n-1} \lambda^i \right| \frac{| 1 - \lambda |}{|\lambda^n|} \ge 1 - \epsilon.
    \end{equation*}
    Thus, the conclusion holds when $N = \max(N_1, N_2)$.
\end{proof}

Now we are ready to characterize the difference between $\hat{\boldsymbol{u}}^{B}$ and $\boldsymbol{u}^{B}$ under the perturbation $\delta u$.

\begin{lem} \label{thm:method:3}
    Under a perturbation $\delta u$ to the solution, there exists $N \in \mathbb{N}^+$ and $\alpha > 0$ such that when $N_t > N$, the solution difference $\| \hat{\boldsymbol{u}}^{B} - \boldsymbol{u}^{B} \|$ by block time-marching strategy satisfies
    \begin{equation*}
        \mathbb{E}_{\delta u} [\| \hat{\boldsymbol{u}}^{B} - \boldsymbol{u}^{B} \|] \ge \alpha \mu \sqrt{N_t} \beta(N_t),
    \end{equation*}
    where 
    \begin{equation*}
        \beta(x) = \begin{cases}
            1, & \mathrm{if}\; \max(|\lambda_k|) \le 1\; \mathrm{and}\; 1 \notin \{ \lambda_k \}, \\
            x, & \mathrm{if}\; \max(|\lambda_k|) = 1\; \mathrm{and}\; 1 \in \{ \lambda_k \}, \\
            | \lambda_{m} |^{x / 2}, & \mathrm{if}\; \max(|\lambda_k|) > 1.
        \end{cases}
    \end{equation*}
\end{lem}

\begin{proof}
    The difference between $\hat{\boldsymbol{u}}_i^{B}$ and $\boldsymbol{u}_i^{B}$ satisfies
    \begin{align*}
        \hat{\boldsymbol{u}}^{B}_{i} - \boldsymbol{u}^{B}_{i} =& \delta u + (\tilde{\mathbf{A}}_i^\top \tilde{\mathbf{A}}_i)^{-1} \tilde{\mathbf{A}}_i^\top \hat{\boldsymbol{b}}_i - \boldsymbol{u}_i^{B} \\
        =& \delta u + (\tilde{\mathbf{A}}_i^\top \tilde{\mathbf{A}}_i)^{-1} \tilde{\mathbf{A}}_i^\top (\tilde{\boldsymbol{b}}_i + [\mathbf{\Phi}_{i-1,1}^\top, \boldsymbol{0}_{(J_n \times Q_tQ_x)}^\top ]^\top (\hat{\boldsymbol{u}}_{i-1}^{B} - \boldsymbol{u}_{i-1}^{B}) ) - \boldsymbol{u}_i^{B} \\
        =& \delta u + \mathbf{B} (\hat{\boldsymbol{u}}_{i-1}^{B} - \boldsymbol{u}_{i-1}^{B}) \\
        =& \delta u + \mathbf{B} (\delta u + \mathbf{B} (\hat{\boldsymbol{u}}_{i-2}^{B} - \boldsymbol{u}_{i-2}^{B})) \\
        =& \cdots \\
        =& \left( \sum_{j=0}^{i-1} \mathbf{B}^{j} \right) \delta u,
    \end{align*}
    where $\hat{\boldsymbol{b}}_i \neq \tilde{\boldsymbol{b}}_i$ due to $\hat{\boldsymbol{u}}_{i-1}^{B} - \boldsymbol{u}_{i-1}^{B} \neq \boldsymbol{0}$. Notice that $\{\boldsymbol{b}_k\}$ is an orthonormal basis of $\mathbb{R}^{J_n}$, thus
    \begin{align*}
        \| \hat{\boldsymbol{u}}^{B} - \boldsymbol{u}^{B} \|^2 =& \sum_{i=1}^{N_t} \| \hat{\boldsymbol{u}}^{B}_{i} - \boldsymbol{u}^{B}_{i} \|^2 & \\
        =& \sum_{i=1}^{N_t} \left\| \sum_{k=1}^{J_n} \sum_{j=0}^{i-1} \delta u_k \lambda_k^j \boldsymbol{b}_k  \right\|^2 & \quad {\scriptstyle (\mathbf{B}^{j} \boldsymbol{b}_k = \lambda_k \mathbf{B}^{j-1} \boldsymbol{b}_k = \cdots = \lambda_k^{j} \boldsymbol{b}_k)} \\
        =& \sum_{i=1}^{N_t} \sum_{k=1}^{J_n} \left| \sum_{j=0}^{i-1} \lambda_k^j \right|^2 \delta u_k ^2 & \quad {\scriptstyle \text{(Unitary orthogonality of $\{\boldsymbol{b}_k\}$)}} \\
        \ge& \sum_{i=1}^{N_t} \left| \sum_{j=0}^{i-1} \lambda_{m}^j \right|^2 \delta u_{m}^2. & \\
    \end{align*}
    The lower bound can be estimated for four different cases.
    \begin{enumerate}
    	\item When $|\lambda_{m}| < 1$, we have
    \begin{align*}
        \mathbb{E}_{\delta u} [ \| \hat{\boldsymbol{u}}^{B} - \boldsymbol{u}^{B} \| ] \ge& \mathbb{E}_{\delta u} \left[ \left( \left( \sum_{i=1}^{N} + \sum_{i=N+1}^{N_t} \right) \left| \sum_{j=0}^{i-1} \lambda_{m}^j \right|^2 \delta u_{m} ^2 \right)^{\frac{1}{2}} \right] & \\
        \ge& \mathbb{E}_{\delta u} \left[ \left( \sum_{i=N+1}^{N_t} (1 - \delta) \frac{C_{\lambda_{m}}^{2i}}{| 1 - \lambda_{m} |^2} \delta u_{m}^2 \right)^{\frac{1}{2}} \right] & \quad {\scriptstyle (\text{Lemma~\ref{lem:method:2}})} \\
        \ge& \alpha \mu \sqrt{N_t}, & \quad {\scriptstyle (N_t - N \ge \frac{N_t}{N+1})}
    \end{align*}
    where $\alpha = \left( \frac{1 - \epsilon}{| 1 - \lambda_{m} |^2 (N + 1)} \right)^{\frac{1}{2}}$.

    \item When $|\lambda_m| = 1$ and $\lambda_m \neq 1$, let $\lambda = e^{j \theta}$, where $j^2 = -1$ and $\theta \in (0, 2\pi)$, then we have
    \begin{align*}
        \mathbb{E}_{\delta u} [ \| \hat{\boldsymbol{u}}^{B} - \boldsymbol{u}^{B} \| ] \ge& \mathbb{E}_{\delta u} \left[ \left( \sum_{i=1}^{N_t} \left| \frac{1 - e^{j i \theta}}{1 - e^{j \theta}} \right|^2 \delta u_{m} ^2 \right)^{\frac{1}{2}} \right] & \\
        =& \mathbb{E}_{\delta u} \left[ \left( \sum_{i=1}^{N_t} \frac{1 - \cos i \theta}{1 - \cos \theta} \right)^{\frac{1}{2}} | \delta u_{m} | \right] & \\
        =& \mathbb{E}_{\delta u} \left[ \frac{1}{(1 - \cos \theta)^{\frac{1}{2}}} \left( N_t - \frac{\sin \frac{N_t + 1}{2} \theta - \sin \frac{1}{2} \theta}{2 \sin \frac{1}{2} \theta} \right)^{\frac{1}{2}} | \delta u_{m} | \right] & \\
        \ge& \mathbb{E}_{\delta u} \left[ \left( \frac{N_t - \frac{1}{\sin \frac{1}{2} \theta}}{1 - \cos \theta} \right)^{\frac{1}{2}} | \delta u_{m} | \right] & \\
        =& \alpha \mu \sqrt{N_t}, & \quad {\scriptstyle (\sqrt{N_t - z} \ge \sqrt{1 - z} \sqrt{N_t} \; \mathrm{for}\; z \in [0, 1))}
    \end{align*}
    where $\alpha = 2 \left( \frac{\sin \frac{1}{2} \theta - 1}{\sin \frac{1}{2} \theta - \sin \frac{3}{2} \theta} \right)^{\frac{1}{2}}$.

    \item When $|\lambda_m| = 1$ and $\lambda_m = 1$, we have
    \begin{align*}
        \mathbb{E}_{\delta u} [ \| \hat{\boldsymbol{u}}^{B} - \boldsymbol{u}^{B} \| ] \ge& \mathbb{E}_{\delta u} \left[ \left( \frac{1}{6} N_t (N_t + 1) (2 N_t + 1) \delta u_{m} ^2 \right)^{\frac{1}{2}} \right] & \quad {\scriptstyle (\sum_{i=1}^{N_t} i^2 = \frac{1}{6} N_t (N_t + 1) (2 N_t + 1))} \\
        \ge& \alpha \mu \sqrt{N_t} N_t,
    \end{align*}
    where $\alpha = 1 / \sqrt{3}$.
    
    \item When $|\lambda_m| > 1$, we have
    \begin{align*}
        \mathbb{E}_{\delta u} [ \| \hat{\boldsymbol{u}}^{B} - \boldsymbol{u}^{B} \| ] \ge& \mathbb{E}_{\delta u} \left[ \left( \left( \sum_{i=1}^{N} + \sum_{i=N+1}^{N_t} \right) \left| \sum_{j=0}^{i-1} \lambda_{m}^j \right|^2 \delta u_{m} ^2 \right)^{\frac{1}{2}} \right] & \\
        \ge& \mathbb{E}_{\delta u} \left[ \left( \sum_{i=N+1}^{N_t} (1 - \epsilon) \frac{C_{\lambda_{m}}^{2i}}{| 1 - \lambda_{m} |^2} \delta u_{m}^2 \right)^{\frac{1}{2}} \right] & \quad {\scriptstyle (\text{Lemma~\ref{lem:method:2}})} \\
        \ge& \frac{\sqrt{1 - \epsilon}}{|1 - \lambda_m|} \mathbb{E}_{\delta u} \left[ \left( (N_t - N) |\lambda_m|^{N_t - N + 1} \delta u_{m}^2 \right)^{\frac{1}{2}} \right] & \quad {\scriptstyle (\text{Arithmetic geometric mean inequality})} \\
        \ge& \alpha \mu \sqrt{N_t} \beta(N_t), & \quad {\scriptstyle (N_t - N \ge \frac{N_t}{N+1})}
    \end{align*}
    where $\alpha = \left( \frac{(1 - \epsilon) | \lambda_{m} |^{1 - N}}{| 1 - \lambda_{m} |^2 (N + 1)} \right)^{\frac{1}{2}}$.
    \end{enumerate}
\end{proof}
If different random feature functions $\{\phi_{i, j}\}$ are chosen over different time subdomains, we observe that $\| \hat{\boldsymbol{u}}^{B} - \boldsymbol{u}^{B} \|$ also satisfies the estimation in Lemma~\ref{thm:method:3}.

\begin{lem} \label{thm:method:4}
    For the space-time random feature method, there exists $N \in \mathbb{N}^+$ and $\alpha > 0$ such that
    \begin{equation*}
        \mathbb{E}_{\delta u} [\| \hat{\boldsymbol{u}}^{S} - \boldsymbol{u}^{S} \|] \le \alpha \delta \sqrt{N_t}.
    \end{equation*}
\end{lem}
\begin{proof}
    \begin{align*}
        (\mathbb{E}_{\delta u} [\| \hat{\boldsymbol{u}}^{S} - \boldsymbol{u}^{S} \|])^2 \le& \mathbb{E}_{\delta u}[\| \hat{\boldsymbol{u}}^{S}_{i} - \boldsymbol{u}^{S}_{i} \|^2] & {\scriptstyle (\text{Schwarz inequality})} \\ 
        =& \mathbb{E}_{\delta u} \left[ \left\| \sum_{i=1}^{N_t} \sum_{k=1}^{M} \delta u_{k,i} \boldsymbol{b}_k \right\|^2 \right] & \\
        =& \alpha \delta^2 N_t,
    \end{align*}
    where $\alpha = \sqrt{M}$.
\end{proof}

\begin{lem} \label{thm:method:5}
    Let two functions $u_{M}(\boldsymbol{x}, t)$ and $\hat{u}_{M}(\boldsymbol{x}, t)$ be of the form~\eqref{eqn:method:tnsol} with $N_x = 1$ and the coefficients being $\boldsymbol{u}$, $\hat{\boldsymbol{u}}$, respectively, then there exists $C_1 \ge C_2 > 0$ such that
    \begin{equation*}
        C_1 \| \hat{\boldsymbol{u}} - \boldsymbol{u} \| \ge \| \hat{u}_{M}(\boldsymbol{x}, t) - u_{M}(\boldsymbol{x}, t) \|_{L^2} \ge C_2 \| \hat{\boldsymbol{u}} - \boldsymbol{u} \|.
    \end{equation*}
\end{lem}
\begin{proof}
    Notice that
    \begin{equation*}
        \| \hat{u}_{M}(\boldsymbol{x}, t) - u_{M}(\boldsymbol{x}, t) \|_{L^{2}} = \left( \sum_{n=1}^{N_t} \sum_{i=1}^{J_n} (\hat{u}_{n,i} - u_{n, i})^2 \int_{\Omega \times [0, T]} \phi_{n, i}(\boldsymbol{x}, t)^2 \; \mathrm{d}\boldsymbol{x} \mathrm{d}t \right)^{\frac{1}{2}} \le C_1 \| \boldsymbol{\hat{u}} - \boldsymbol{u} \|,
    \end{equation*}
    where $C_1 = \max_{n,i} \| \phi_{n, i}(\boldsymbol{x}, t) \|_2$, and
    \begin{equation*}
        \| \hat{u}_{M}(\boldsymbol{x}, t) - u_{M}(\boldsymbol{x}, t) \|_{L^{2}} = \left( \sum_{n=1}^{N_t} \sum_{i=1}^{J_n} (\hat{u}_{n,i} - u_{n, i})^2 \int_{\Omega \times [0, T]} \phi_{n, i}(\boldsymbol{x}, t)^2 \; \mathrm{d}\boldsymbol{x} \mathrm{d}t \right)^{\frac{1}{2}} \ge C_2 \| \boldsymbol{\hat{u}} - \boldsymbol{u} \|,
    \end{equation*}
    where $C_2 = \min_{n,i} \| \phi_{n, i}(\boldsymbol{x}, t) \|_2$. Moreover, for all $\phi_{n, i} \in \mathcal{L}^2(\Omega \times [0, T])$, we have $C_1, C_2 < +\infty$.
\end{proof}

\begin{thm}[Lower bound] \label{coro:method:3}
    Denote $u_{e}$ the exact solution. Let $\delta u$ be the error between the numerical solution and the exact solution to the least-squares problem. Given $\epsilon > 0$, we assume that
    \begin{equation*}
    \| u^{B}_{M}(\boldsymbol{x}, t) - u_e(\boldsymbol{x}, t) \|_{L^2(\Omega\times[t_n,t_{n+1}])} \leq \epsilon, \quad\forall\; n=0,\cdots, N.
    \end{equation*}
	Then, there exists $N \in \mathbb{N}^+$ and $\alpha > 0$ such that when $N_t > N$,
	\begin{equation}
        \mathbb{E}_{\delta u} [\| \hat{u}^{B}_{M}(\boldsymbol{x}, t) - u_e(\boldsymbol{x}, t) \|_{L^{2}}] \ge \alpha \mu \sqrt{N_t} \beta(N_t) - \epsilon \sqrt{N_t}.
        \label{eqn:coro:method:3}
    \end{equation}
\end{thm}
\begin{proof}
    Notice that
    \begin{align*}
        \mathbb{E}_{\delta u} [\| \hat{u}^{B}_{M}(\boldsymbol{x}, t) - u_e(\boldsymbol{x}, t) \|_{L^{2}}] \ge& \mathbb{E}_{\delta u} [\| \hat{u}^{B}_{M}(\boldsymbol{x}, t) - u^{B}_{M}(\boldsymbol{x}, t) \|_{L^{2}}] - \| u^{B}_{M}(\boldsymbol{x}, t) - u_e(\boldsymbol{x}, t) \|_{L_{2}} & \quad {\scriptstyle (\text{Triangle inequality})} \\
        \ge& \alpha \mu \sqrt{N_t} \beta(N_t) - \left( \sum_{n=1}^{N_t} \epsilon^2 \right)^{\frac{1}{2}} & \quad {\scriptstyle (\text{Lemma~\ref{thm:method:3} and Lemma~\ref{thm:method:5}})} \\
        =& \alpha \mu \sqrt{N_t} \beta(N_t) - \epsilon \sqrt{N_t}. &
    \end{align*}
\end{proof}

\begin{thm}[Upper bound] \label{coro:method:4}
    Denote $u_{e}$ the exact solution. Let $\delta u$ be the error between the numerical solution and the exact solution to the least-squares problem. Given $\epsilon > 0$, we assume that
	\begin{equation*}
		\| u^{S}_{M}(\boldsymbol{x}, t) - u_e(\boldsymbol{x}, t) \|_{L^2(\Omega\times[t_n,t_{n+1}])} \leq \epsilon, \quad\forall\; n=0,\cdots, N.
	\end{equation*}
	Then, there exists $N \in \mathbb{N}^+$ and $\alpha > 0$ such that when $N_t > N$,
    \begin{equation}
        \mathbb{E}_{\delta u} [\| \hat{u}^{S}(\boldsymbol{x}, t) - u_e(\boldsymbol{x}, t) \|_{L^{2}}] \le \alpha \delta \sqrt{N_t} + \epsilon \sqrt{N_t}.
        \label{eqn:coro:method:4}
    \end{equation}
\end{thm}
\begin{proof}
    Notice that
    \begin{align*}
        \mathbb{E}_{\delta u} [\| \hat{u}^{S}(\boldsymbol{x}, t) - u_e(\boldsymbol{x}, t) \|_{L^{2}}] \le& \mathbb{E}_{\delta u} [\| \hat{u}^{S}(\boldsymbol{x}, t) - u^{S}(\boldsymbol{x}, t) \|_{L^{2}}] + \| u^{S}(\boldsymbol{x}, t) + u_e(\boldsymbol{x}, t) \|_{L_{2}} & \quad {\scriptstyle (\text{Triangle inequality})} \\
        \le& \alpha \sigma \sqrt{N_t} + \left( \sum_{n=1}^{N_t} \epsilon^2 \right)^{\frac{1}{2}} & \quad {\scriptstyle (\text{Lemma~\ref{thm:method:4} and Lemma~\ref{thm:method:5}})} \\
        =& \alpha \sigma \sqrt{N_t} + \sqrt{N_t} \epsilon. &
    \end{align*}
\end{proof}

From Theorem~\ref{coro:method:3} and Theorem~\ref{coro:method:4}, we see that the error of solving the least-squares problem by the block time-marching strategy increases exponentially in time, while the error in the ST-RFM does not have this problem. These are also confirmed by the numerical results given below.

\section{Numerical Results} \label{sec:num}

In this section, we present numerical results for one-dimensional and two-dimensional problems with simple geometry and a two-dimensional problem with complex geometry to demonstrate the effectiveness of the ST-RFM and confirm theoretical results.

\subsection{One-dimensional Problems} \label{sec:num:oned}

\subsubsection{Heat Equation} \label{sec:num:oned:heat}

Consider the following problem
\begin{equation} 
\left\{
\begin{aligned}
& \partial_t u(x, t) - \alpha^2 \partial_x^2 u(x, t) = 0, \qquad & x \in [x_0, x_1], t \in [0, T], \\
& u(x_0, t) = g_1(t), & t \in [0, T], \\
& u(x_1, t) = g_2(t), & t \in [0, T], \\
& u(x, 0) = h(x), & x \in [x_0, x_1], \\
\end{aligned}
\right.
\label{eqn:num:heat}
\end{equation}
where $ \alpha = \pi / 2$, $x_0 = 0$, $x_1 = 12$ and $T=10$. The exact solution is chosen to be
\begin{equation}
u_e(x, t) = 2 \sin(\alpha x) e^{-t}. \label{eqn:num:heat:sol}
\end{equation}
We choose the initial condition $h(x)$ by restricting Eq.~\eqref{eqn:num:heat:sol} to $t=0$, and the boundary conditions $g_1(t)$ and $g_2(t)$ by restricting Eq.~\eqref{eqn:num:heat:sol} to $x=x_0$ and $x=x_1$, respectively.

Set the default hyper-parameters $N_x=2$, $N_t=5$, $Q_x=20$, $Q_t=20$, $J_n=400$ and $N_b=1$. Numerical solutions and errors of STC (Eq.~\eqref{eqn:method:concat}) and SoV (Eq.~\eqref{eqn:method:sov}) are plotted in Figure~\ref{exp:fig:heat}. The $L^{\infty}$  error in the ST-RFM is small ($< 4.5e-6$), which indicates that both random feature functions have strong approximation properties.
\begin{figure}[htbp]
	\centering
	\subfigure[STC.]{
		\includegraphics[width=.22\linewidth]{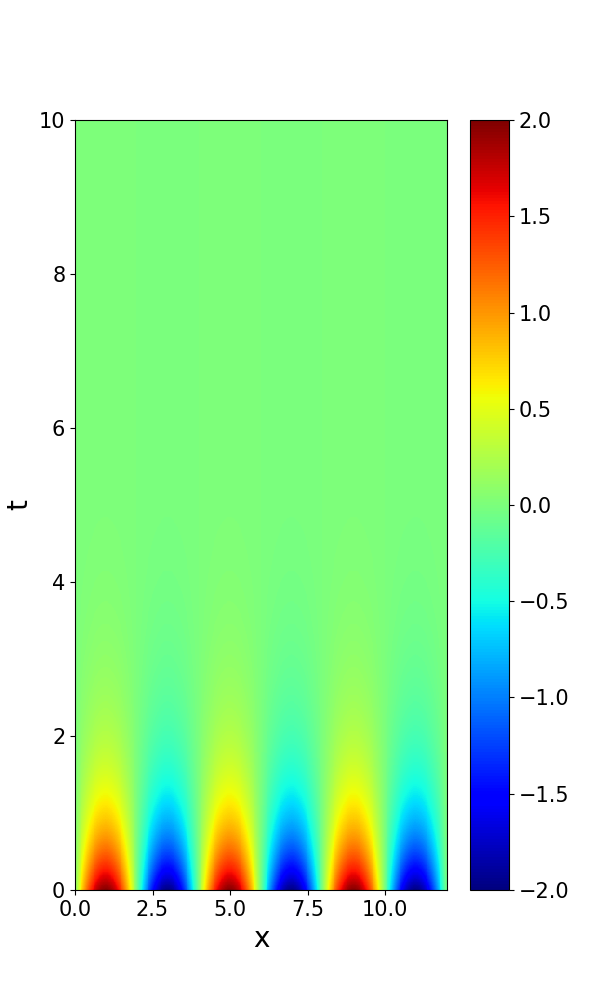}
		\includegraphics[width=.22\linewidth]{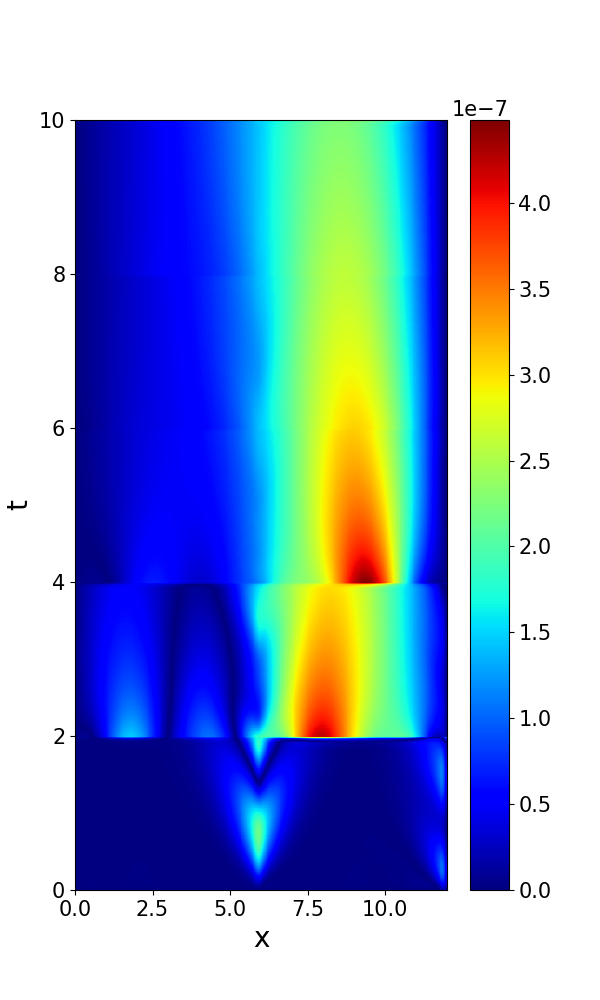}
	}
	\subfigure[SoV.]{
		\includegraphics[width=.22\linewidth]{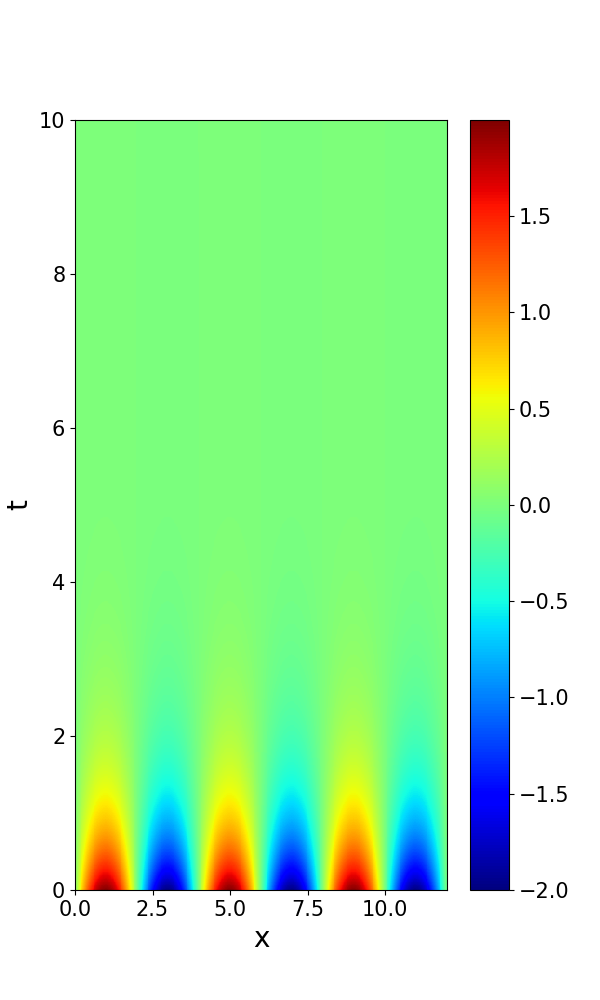}
		\includegraphics[width=.22\linewidth]{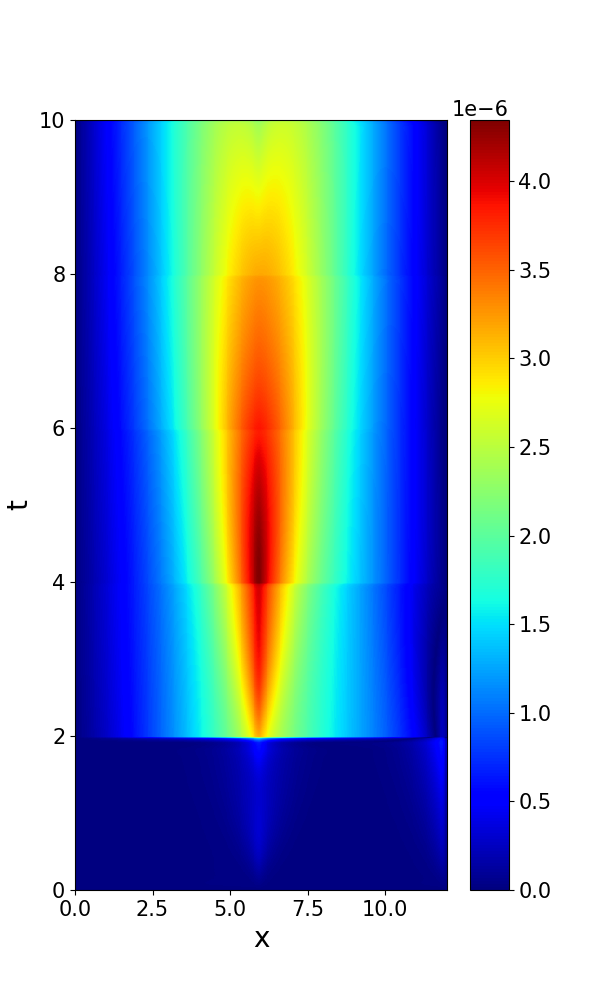}
	}
	\caption{Results for the heat equation. (a) STC: Left, numerical solution; right, absolute error. (b) SoV: Left, numerical solution; right, absolute error.}
	\label{exp:fig:heat}
\end{figure}

Next, we report the convergence behavior with respect to different parameters in Figure~\ref{exp:fig:heat1}(a)-(d). In Figure~\ref{exp:fig:heat1:a}, we set $N_b=1$, $N_x=2$, $J_n=400$, $Q_x=Q_t=20$ and $N_t=1,\cdots,5$ to verify the convergence with respect to $N_t$. In Figure~\ref{exp:fig:heat1:b}, we set $N_x=2$, $N_t=1$, $J_n=400$, $Q_x=Q_t=20$ and $N_b=1,\cdots,5$ to verify the convergence with respect to $N_b$. In Figure~\ref{exp:fig:heat1:c}, we set $N_b=5$, $N_x=2$, $N_t=1$, $Q_x=Q_t=20$ and $J_n=50,100,200,300,400$ to verify the convergence with respect to $J_n$. In Figure~\ref{exp:fig:heat1:d}, we set $N_b=5$, $N_x=2$, $N_t=1$, $J_n=400$ and $Q_x=Q_t=5,10,15,20,25$ to verify the convergence with respect to $Q_x$/$Q_t$. A clear trend of spectral accuracy is observed for the ST-RFM in both spatial and temporal directions.

Now, we compare STC and SoV in Figure~\ref{exp:fig:heat1:e}, where the default hyper-parameter setting is used. For this example, STC performs better than SoV. The comparison between the block time-marching strategy and the ST-RFM is plotted in Figure~\ref{exp:fig:heat1:f}, where we set $N_b=5$ and $N_t=1$ for the block time-marching strategy and $N_b=1$ and $N_t=5$ for the ST-RFM. The $L^2$ error of solution by the block time-marching strategy increases exponentially fast with respect to the number of blocks, while the error in the ST-RFM remains almost flat over all time subdomains.

\begin{figure}[htbp]
	\centering
	\subfigure[Convergence w.r.t. $N_t$.\label{exp:fig:heat1:a}]{
		\includegraphics[width=.30\linewidth]{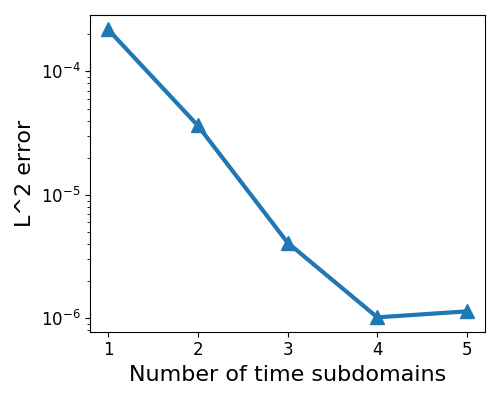}
	}
	\subfigure[Convergence w.r.t. $N_b$.\label{exp:fig:heat1:b}]{
		\includegraphics[width=.30\linewidth]{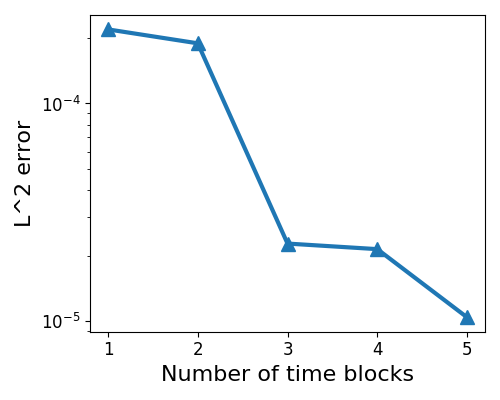}
	}
	\subfigure[Convergence w.r.t. $J_n$.\label{exp:fig:heat1:c}]{
		\includegraphics[width=.30\linewidth]{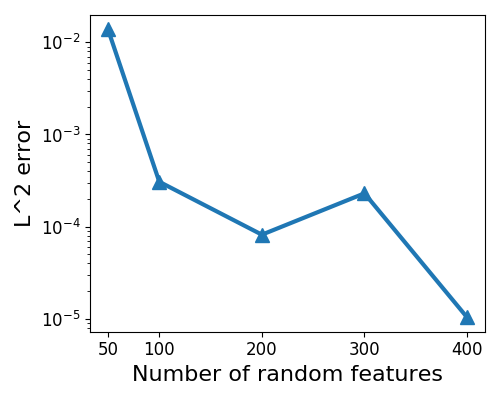}
	}
	\subfigure[Convergence w.r.t. $Q$.\label{exp:fig:heat1:d}]{
		\includegraphics[width=.30\linewidth]{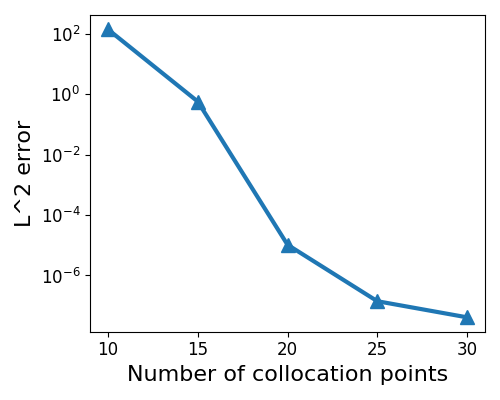}
	}
	\subfigure[STC v.s. SoV.\label{exp:fig:heat1:e}]{
		\includegraphics[width=.30\linewidth]{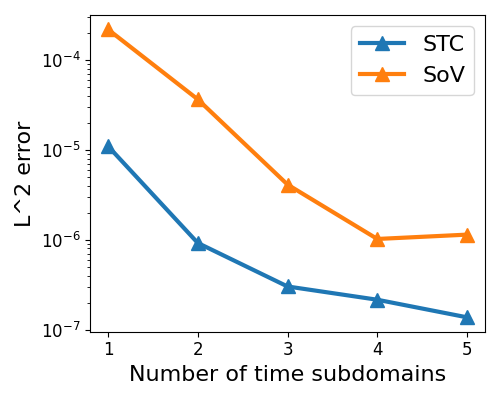}
	}
	\subfigure[Block time-marching v.s. ST-RFM.\label{exp:fig:heat1:f}]{
		\includegraphics[width=.30\linewidth]{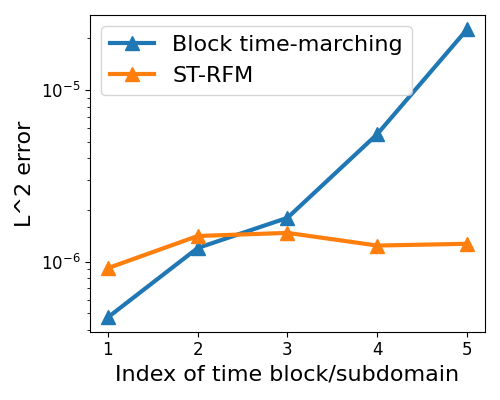}
	}
	\caption{Convergence behavior as different hyper-parameters are varied for the heat equation. (a) the number of time subdomains in the ST-RFM; (b) the number of time blocks in the block time-marching strategy; (c) the number of random features in each direction; (d) the number of collocation points in each direction; (e) comparison of STC and SoV; (f) comparison of the ST-RFM and the block time-marching strategy.}
	\label{exp:fig:heat1}
\end{figure}

\subsubsection{Heat Equation with Nonsmooth Initial Condition} \label{sec:num:oned:heatnonsmooth}

Consider the heat equation~\eqref{eqn:num:heat} with $x_0 = 0$, $x_1 = 8$ and the nonsmooth initial condition as follows:
\begin{equation}
    h(x) = 2 \mathds{I}_{0 \le x < 4} \sin \left( \pi x / 2 \right) + 2 \mathds{I}_{4 \le x \le 8} \sin \left( {\pi} x \right).
    \label{eqn:num:heatnonsmooth:initial}
\end{equation}
It is easy to check that $h(x)$ only belongs to $C([0, 8])$.

Since an exact solution is not analytically available for the problem under consideration, we employ a set of hyper-parameters, namely $N_{x}=2$, $N_{t}=5$, $Q_{x}=30$, $Q_{t}=50$, $J_{n}=250$, and $N_{b}=1$, to obtain a reference solution. Numerical convergence of ST-RFM is evaluated in terms of the relative $L^{2}$ error, which is depicted in Figure~\ref{exp:fig:heat-nonsmooth}. We use STC to implement the basis function of ST-RFM here. Apparently, spectral accuracy is observed. However, the accuracy is only around $10^{-3}$, two orders of magnitude larger than that in the smooth case~\eqref{eqn:num:heat}, which is attributed to the regularity deficiency.

\begin{figure}
    \centering
    \subfigure[Convergence w.r.t. time subdomains.]{
        \includegraphics[width=.3\linewidth]{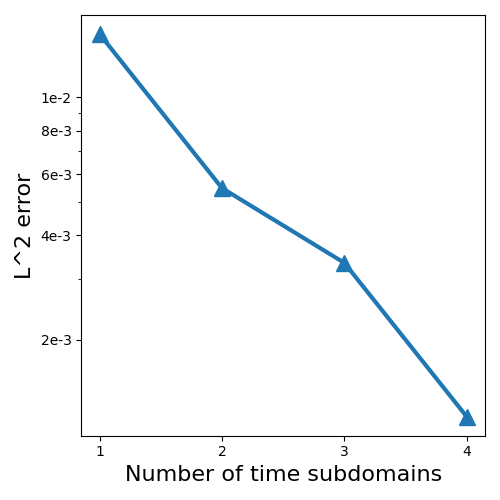}
    }
    \subfigure[Convergence w.r.t. $J_{n}$.]{
        \includegraphics[width=.3\linewidth]{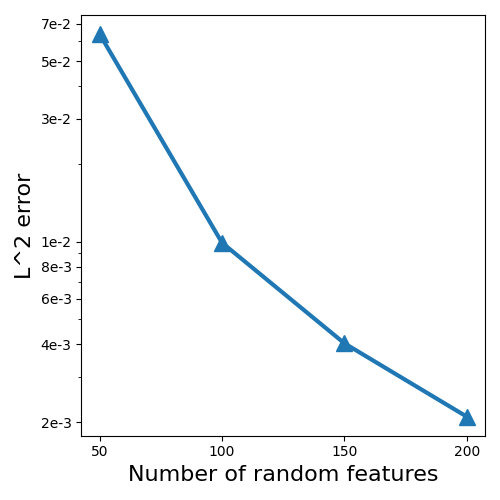}
    }
    \subfigure[Convergence w.r.t. $Q$.]{
        \includegraphics[width=.3\linewidth]{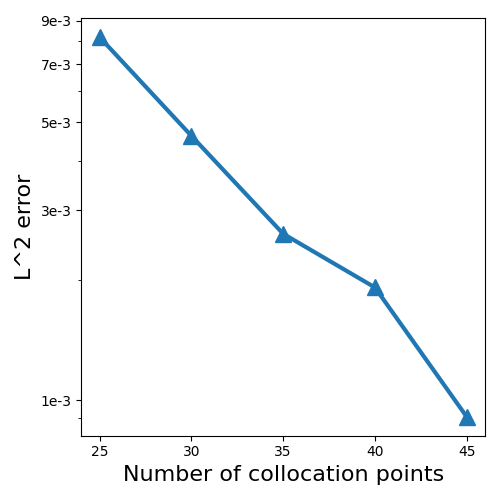}
    }
	\caption{Convergence behavior as different hyper-parameters are varied for the heat equation with nonsmooth initial condition. (a) the number of time subdomains in the ST-RFM; (b) the number of random features in each direction; (c) the number of collocation points in the temporal direction.}
	\label{exp:fig:heat-nonsmooth}
\end{figure}

\subsubsection{Wave Equation} \label{sec:num:oned:wave}

Consider the following problem
\begin{equation}
\left\{
\begin{aligned}
& \partial_t^2 u(x, t) - \alpha^2 \partial_x^2 u(x, t) = 0, & \quad x, t \in [x_0, x_1] \times [0, T], \\
& u(x_0, t) = u(x_1, t) = 0. & \quad t \in [0, T], \\
& u(x, 0) = g_1(x), & \quad x \in [x_0, x_1], \\
& \partial_t u(x, 0) = g_2(x), & \quad x \in [x_0, x_1].  \\
\end{aligned} \label{eqn:wave}
\right.
\end{equation}
where $ x_0 = 0$, $x_1 = 6\pi$, $\alpha = 1 $ and $T=10$. The exact solution is chosen to be
\begin{equation}
u_e(x, t) = \cos \left(\frac{a \pi}{l}t \right) \sin \left( \frac{\pi}{l}x \right) + \left[ \cos \left( \frac{2a \pi}{l} t \right) + \frac{l}{2a \pi} \sin \left( \frac{2a \pi}{l}t \right) \right] \sin \left( \frac{2 \pi}{l}x \right),\quad l = x_1 - x_0. \label{eqn:wave:sol}
\end{equation}
Initial conditions $g_1(x)$ and $g_2(x)$ are chosen accordingly.

Set the default hyper-parameters $N_x=5$, $N_t=5$, $Q_x=30$, $Q_t=30$, $J_n=300$ and $N_b=1$. Numerical solutions and errors of STC and SoV are plotted in Figure~\ref{exp:fig:wave}. The $L^{\infty}$ error is smaller than $ 2.5 e-8$.
\begin{figure}[htbp]
	\centering
	\subfigure[STC.]{
		\includegraphics[width=.22\linewidth]{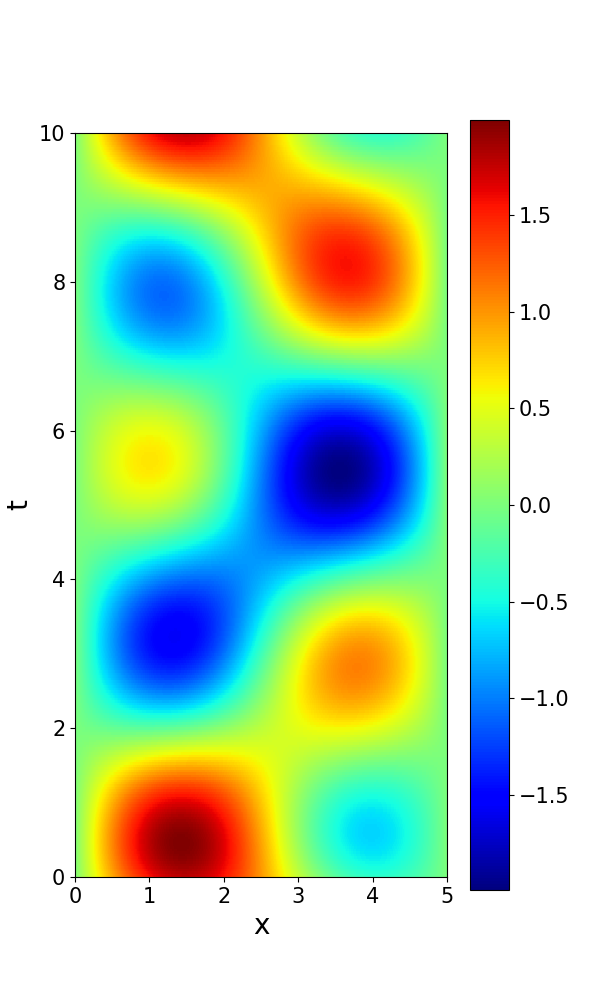}
		\includegraphics[width=.22\linewidth]{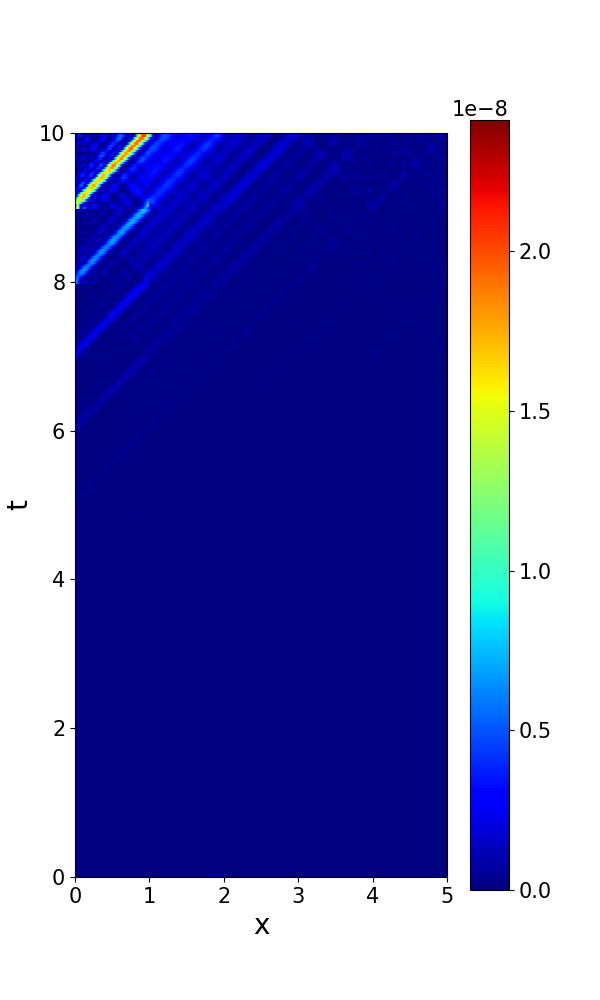}
	}
	\subfigure[SoV.]{
		\includegraphics[width=.22\linewidth]{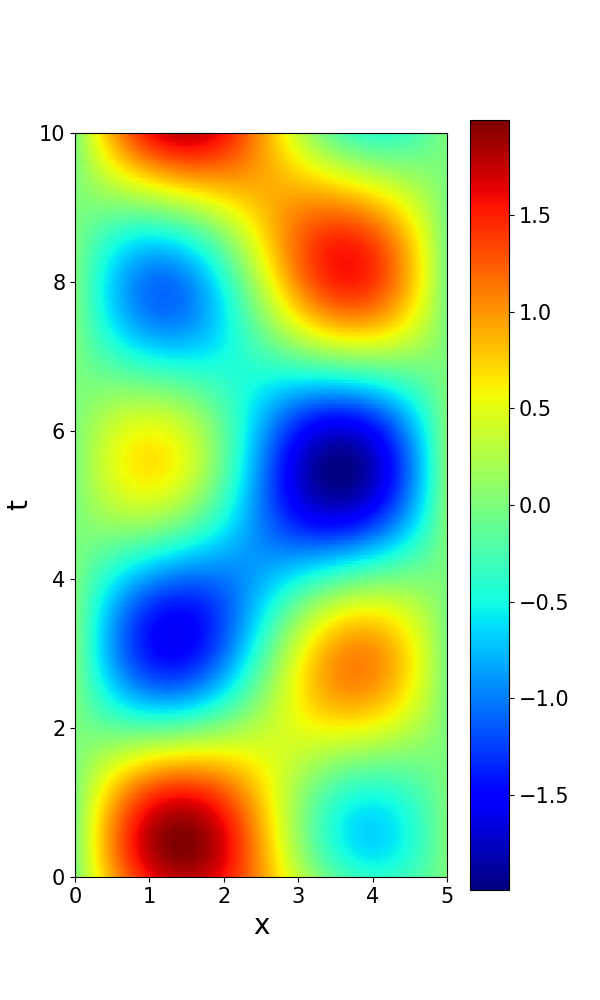}
		\includegraphics[width=.22\linewidth]{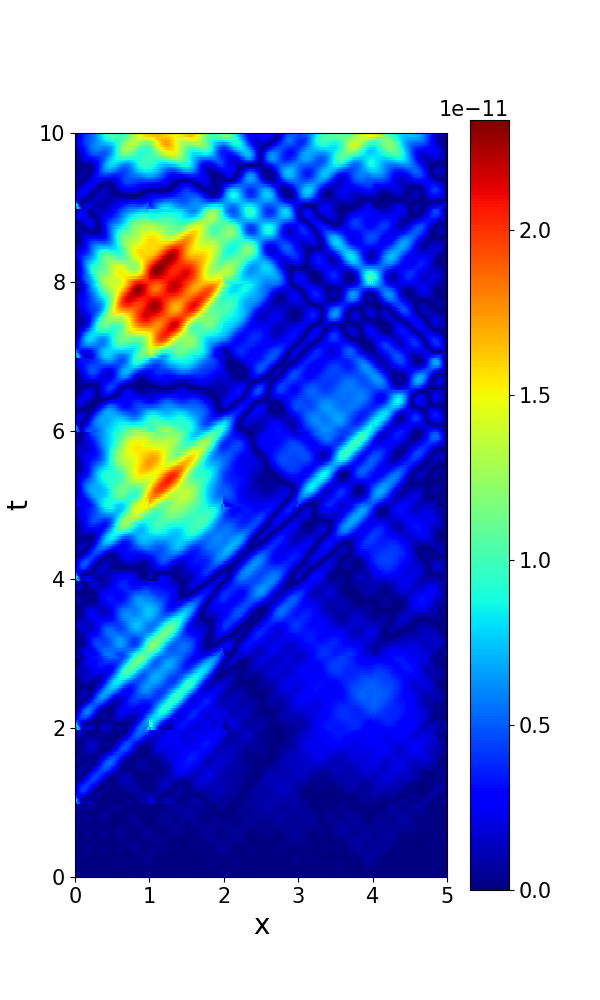}
	}
	\caption{Results for the wave equation. (a) ST-RFM with STC: Left, numerical solution; right, absolute error. (b) ST-RFM with SoV: Left, numerical solution; right, absolute error.}
	\label{exp:fig:wave}
\end{figure}

Next, we report the convergence behavior with respect to different parameters in Figure~\ref{exp:fig:wave1}(a-d). In Figure~\ref{exp:fig:wave1:a}, we set $N_b=1$, $N_x=5$, $J_n=300$, $Q_x=Q_t=30$ and $N_t=1,\cdots,5$ to verify the convergence with respect to $N_t$. In Figure~\ref{exp:fig:wave1:b}, we set $N_x=5$, $N_t=1$, $J_n=300$, $Q_x=Q_t=30$ and $N_b=1,\cdots,5$ to verify the convergence with respect to $N_b$. In Figure~\ref{exp:fig:wave1:c}, we set $N_b=5$, $N_x=5$, $N_t=1$, $Q_x=Q_t=30$ and $J_n=100,150,200,250,300$ to verify the convergence with respect to $J_n$. In Figure~\ref{exp:fig:wave1:d}, we set $N_b=5$, $N_x=5$, $N_t=1$, $J_n=300$ and $Q_x=Q_t=10,15,20,25,30$ to verify the convergence with respect to $Q_x$/$Q_t$. A clear trend of spectral accuracy is observed for the ST-RFM in both spatial and temporal directions. 

Now, we compare STC and SoV in Figure~\ref{exp:fig:wave1:e}, where we set $N_b=1$, $N_x=5$, $N_t=1, \cdots, 5$, $Q_x=Q_t=30$ and $J_n=300$. For this example, SoV performs better than STC. The comparison between the block time-marching strategy and the ST-RFM is plotted in Figure~\ref{exp:fig:wave1:f}, where we set $N_b=5$ and $N_t=1$ for block time-marching strategy and $N_b=1$ and $N_t=5$ for ST-RFM. The $L^2$ error of the solution by the block time-marching strategy increases exponentially fast with respect to the number of blocks, while the error in the ST-RFM remains almost flat over all time subdomains.

\begin{figure}[htbp]
	\centering
	\subfigure[Convergence w.r.t. $N_t$.\label{exp:fig:wave1:a}]{
		\includegraphics[width=.30\linewidth]{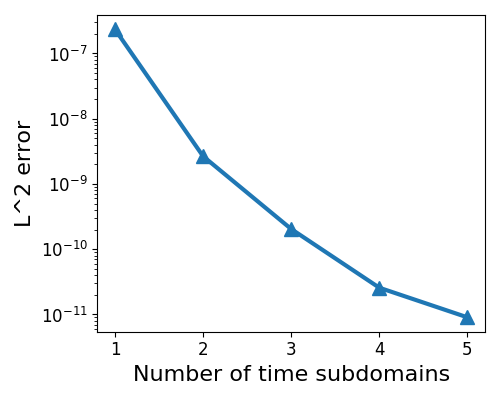}
	}
	\subfigure[Convergence w.r.t. $N_b$.\label{exp:fig:wave1:b}]{
		\includegraphics[width=.30\linewidth]{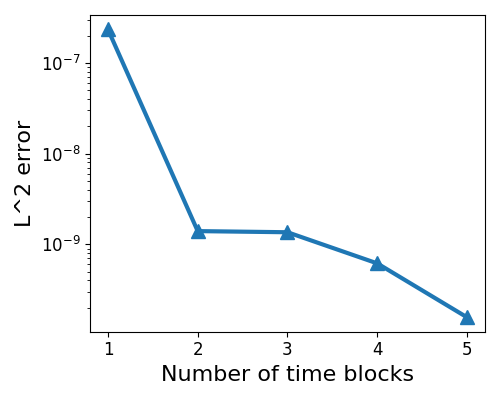}
	}
	\subfigure[Convergence w.r.t. $J_n$.\label{exp:fig:wave1:c}]{
		\includegraphics[width=.30\linewidth]{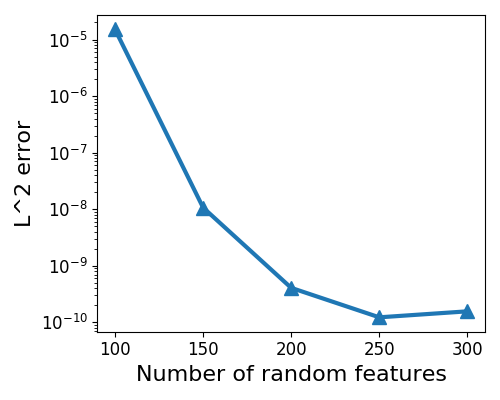}
	}
	\subfigure[Convergence w.r.t. $Q$.\label{exp:fig:wave1:d}]{
		\includegraphics[width=.30\linewidth]{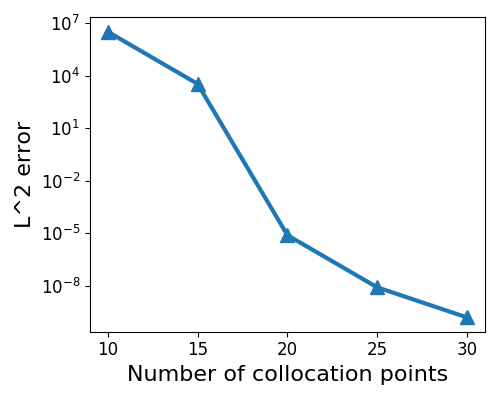}
	}
	\subfigure[STC v.s. SoV.\label{exp:fig:wave1:e}]{
		\includegraphics[width=.30\linewidth]{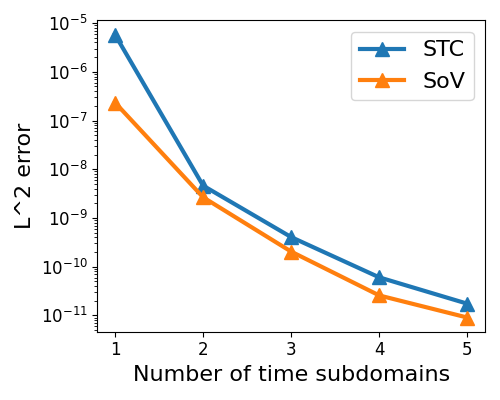}
	}
	\subfigure[Block time-marching v.s. ST-RFM.\label{exp:fig:wave1:f}]{
		\includegraphics[width=.30\linewidth]{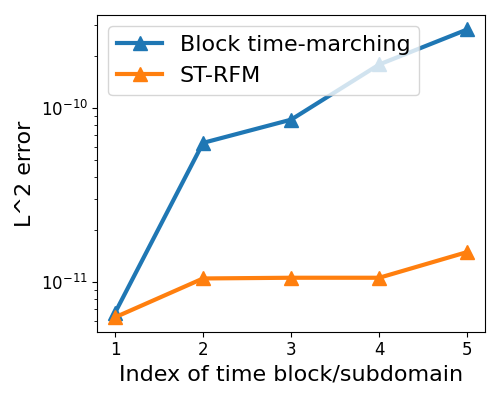}
	}
	\caption{Convergence behavior as different hyper-parameters are varied for the wave equation. (a) the number of time subdomains in the ST-RFM; (b) the number of time blocks in the block time-marching strategy; (c) the number of random features in each direction; (d) the number of collocation points in each direction; (e) comparison of STC and SoV; (f) comparison of the ST-RFM and the block time-marching strategy.}
	\label{exp:fig:wave1}
\end{figure}

\subsubsection{Schr\"{o}dinger Equation} \label{sec:num:oned:schro}

Consider the following problem
\begin{equation}
\left\{
\begin{aligned}
& i \partial_t \psi(x, t) + 0.5 \Delta \psi(x, t) = 0, & \quad x, t \in [x_0, x_1] \times [0, T], \\
& \psi(x, 0) = g(x), & \quad x \in [x_0, x_1], \\
& \psi(x_0, t) = \psi(x_1, t), & \quad t \in [0, T], \\
& \partial_x \psi(x_0, t) = \partial_x \psi(x_1, t) & \quad t \in [0, T],
\end{aligned} \label{eqn:schordinger}
\right.
\end{equation}
where $ x_0 = 0$, $x_1 = 5$ and $T=10$. The exact solution is chosen to be
\begin{equation}
\psi(x, t) = e^{-i\omega^2 t / 2} (2 \cos (\omega x) + \sin (\omega x)), \quad \omega = \frac{2 \pi}{b_1 - a_1}, \label{eqn:schordinger:sol}
\end{equation}
and $g(x)$ is chosen accordingly.

Set the default hyper-parameters $N_x=5$, $N_t=3$, $Q_x=30$, $Q_t=30$, $J_n=300$ and $N_b=1$. Numerical solutions and errors of STC and SoV are plotted in Figure~\ref{exp:fig:schrodinger}. The $L^{\infty}$ error in the ST-RFM is smaller than $2.0 e-9$.
\begin{figure}[htbp]
	\centering
	\subfigure[Real part: STC.]{
		\includegraphics[width=.22\linewidth]{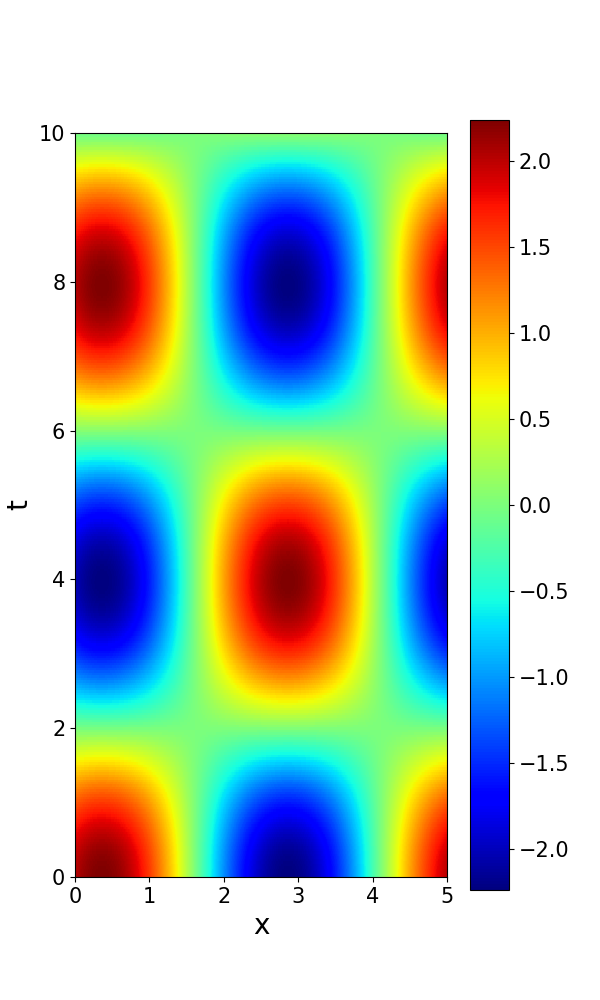}
		\includegraphics[width=.22\linewidth]{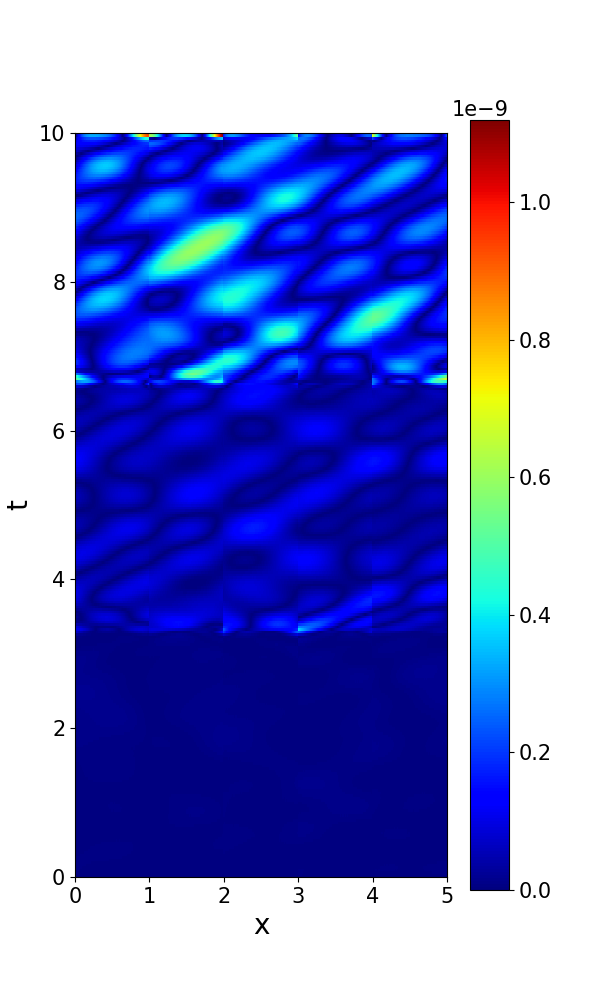}
	}
	\subfigure[Real part: SoV.]{
		\includegraphics[width=.22\linewidth]{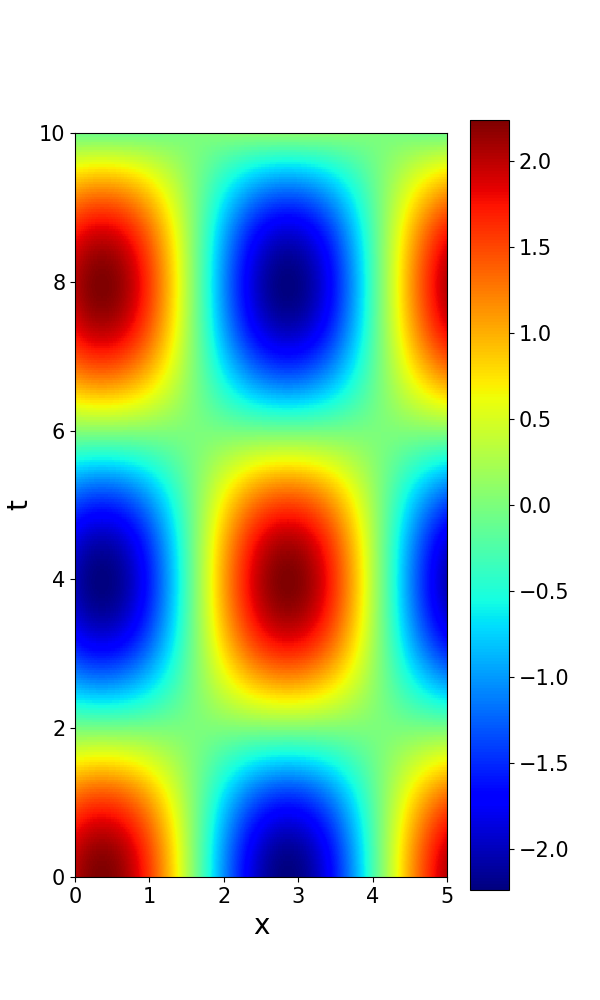}
		\includegraphics[width=.22\linewidth]{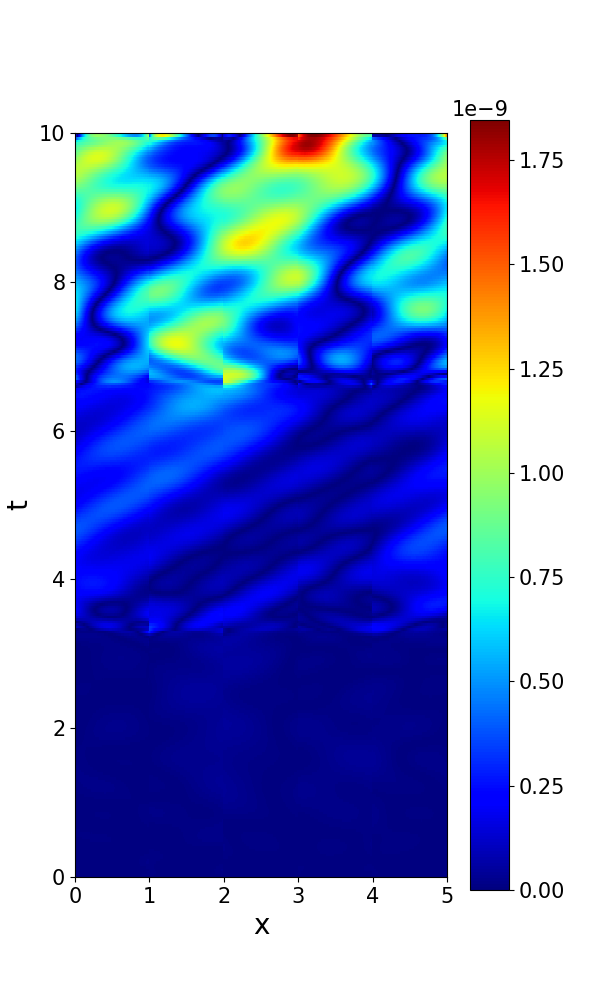}
	}
	\subfigure[Imaginary part: STC.]{
		\includegraphics[width=.22\linewidth]{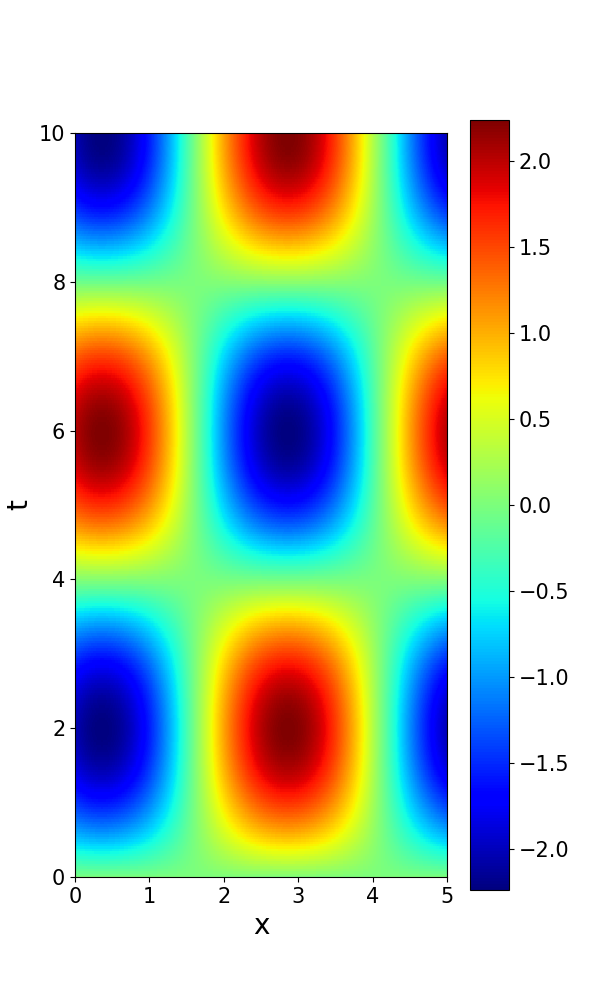}
		\includegraphics[width=.22\linewidth]{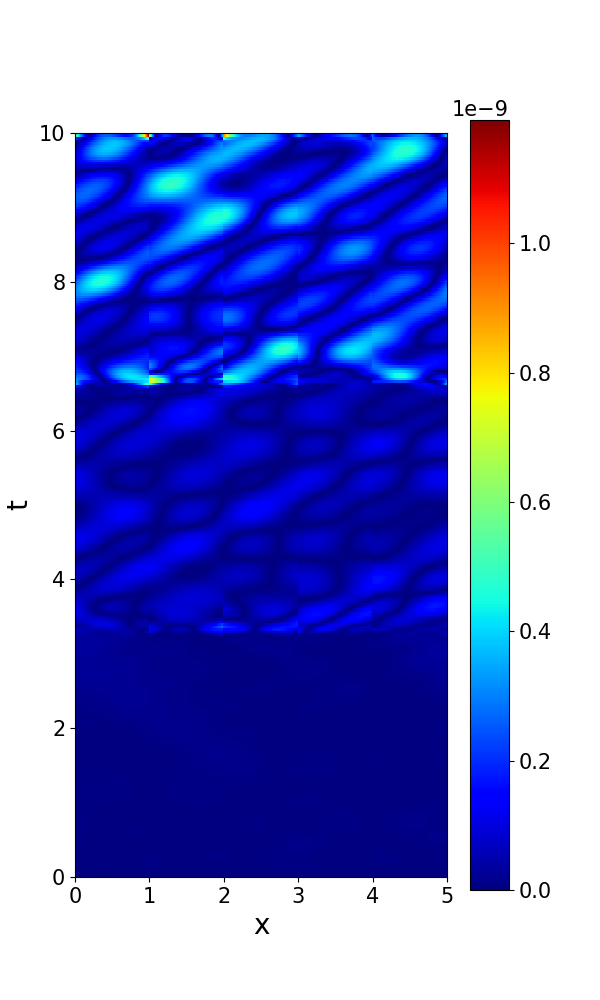}
	}
	\subfigure[Imaginary part: SoV.]{
		\includegraphics[width=.22\linewidth]{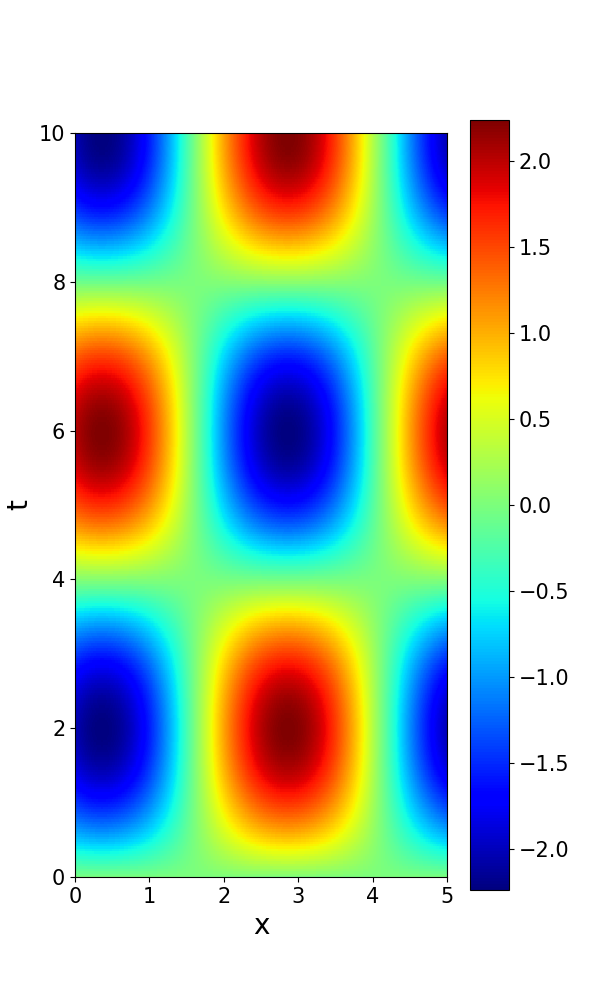}
		\includegraphics[width=.22\linewidth]{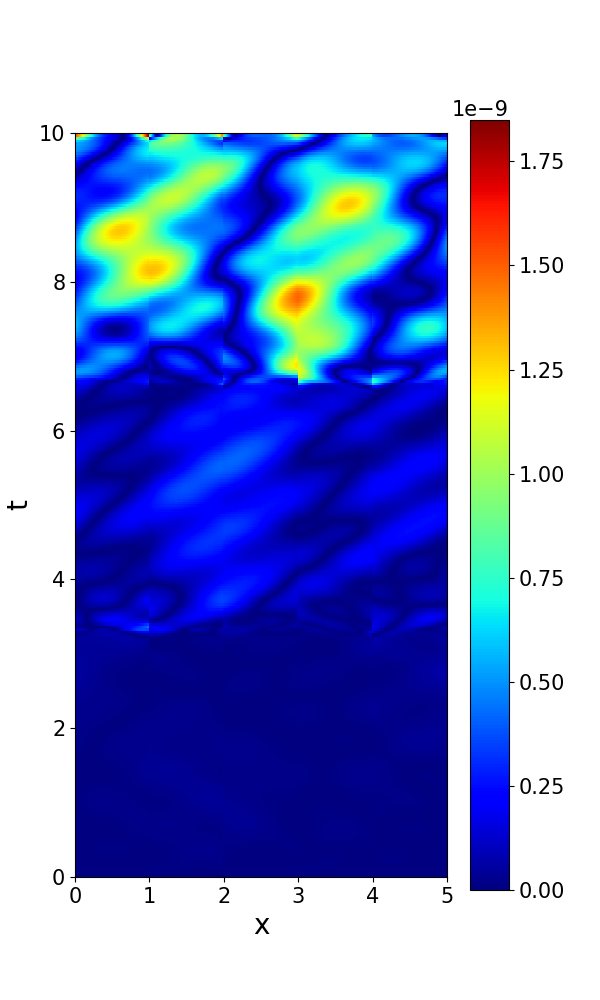}
	}
	\caption{Results for Schr\"{o}dinger equation. (a), (c) STC: Left, numerical solution; right, absolute error; (b), (d) SoV: Left, numerical solution; right, absolute error. }
	\label{exp:fig:schrodinger}
\end{figure}

Next, we report the convergence behavior with respect to different parameters in Figure~\ref{exp:fig:sch1:ur}(a-d) (real part) and Figure~\ref{exp:fig:sch2:ui}(a-d) (imaginary part). In Figure~\ref{exp:fig:sch1:ur:a} and Figure~\ref{exp:fig:sch2:ui:a}, we set $N_b=1$, $N_x=5$, $J_n=300$, $Q_x=Q_t=30$ and $N_t=1,2,3$ to verify the convergence with respect to $N_t$. In Figure~\ref{exp:fig:sch1:ur:b} and Figure~\ref{exp:fig:sch2:ui:b}, we set $N_x=5$, $N_t=1$, $J_n=300$, $Q_x=Q_t=30$ and $N_b=1,2,3$ to verify the convergence with respect to $N_b$. In Figure~\ref{exp:fig:sch1:ur:c} and Figure~\ref{exp:fig:sch2:ui:c}, we set $N_b=3$, $N_x=5$, $N_t=1$, $Q_x=Q_t=30$ and $J_n=100,150,200,250,300$ to verify the convergence with respect to $J_n$. In Figure~\ref{exp:fig:sch1:ur:d} and Figure~\ref{exp:fig:sch2:ui:d}, we set $N_b=3$, $N_x=5$, $N_t=1$, $J_n=300$ and $Q_x=Q_t=10,15,20,25,30$ to verify the convergence with respect to $Q_x$/$Q_t$. A clear trend of spectral accuracy is observed for the ST-RFM in both spatial and temporal directions.

Now, we compare STC and SoV in Figure~\ref{exp:fig:sch1:ur:e} and Figure~\ref{exp:fig:sch2:ui:e}, where we set $N_b=1$, $N_x=5$, $N_t=1, 2, 3$, $Q_x=Q_t=30$ and $J_n=300$. For this example, SoV performs better than STC. The comparison between the block time-marching strategy and the ST-RFM is plotted in Figure~\ref{exp:fig:sch1:ur:f} and Figure~\ref{exp:fig:sch2:ui:f}, where we set $N_b=3$ and $N_t=1$ for block time-marching strategy and $N_b=1$ and $N_t=3$ for ST-RFM. The $L^2$ error of the solution by the block time-marching strategy increases exponentially fast with respect to the number of blocks, while the error in the ST-RFM increases slower over all time subdomains.

\begin{figure}[htbp]
	\centering
	\subfigure[Convergence w.r.t. $N_t$.\label{exp:fig:sch1:ur:a}]{
		\includegraphics[width=.30\linewidth]{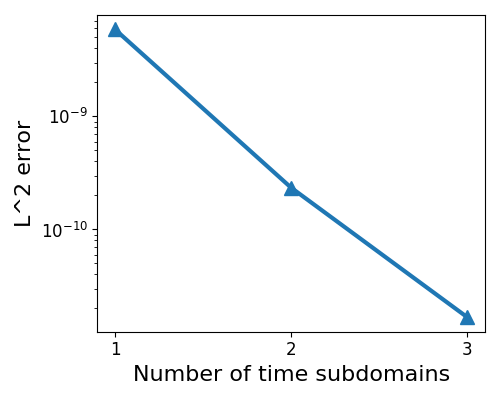}
	}
	\subfigure[Convergence w.r.t. $N_b$.\label{exp:fig:sch1:ur:b}]{
		\includegraphics[width=.30\linewidth]{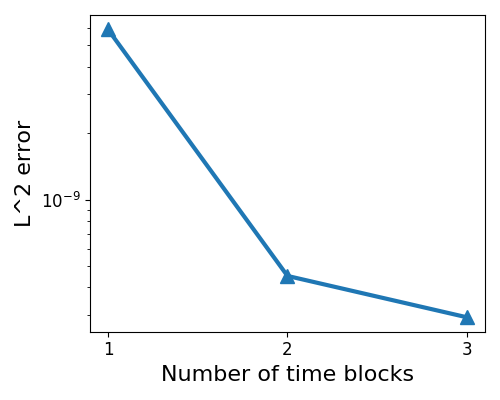}
	}
	\subfigure[Convergence w.r.t. $J_n$.\label{exp:fig:sch1:ur:c}]{
		\includegraphics[width=.30\linewidth]{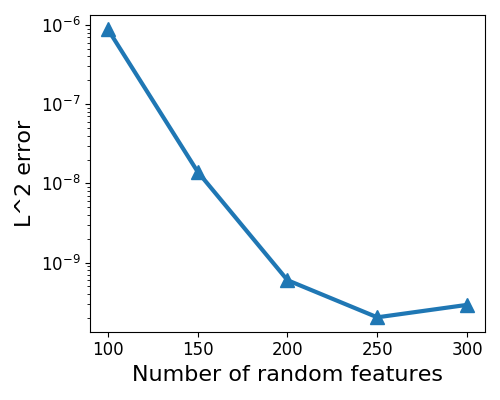}
	}
	\subfigure[Convergence w.r.t. $Q$.\label{exp:fig:sch1:ur:d}]{
		\includegraphics[width=.30\linewidth]{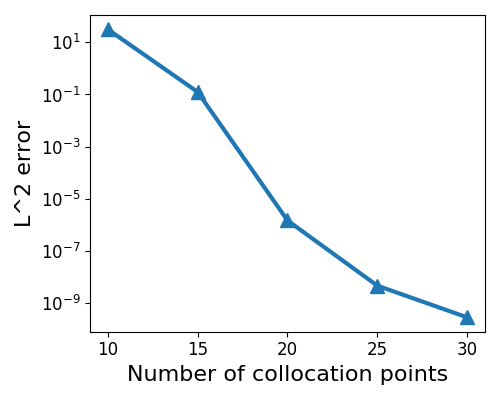}
	}
	\subfigure[STC v.s. SoV.\label{exp:fig:sch1:ur:e}]{
		\includegraphics[width=.30\linewidth]{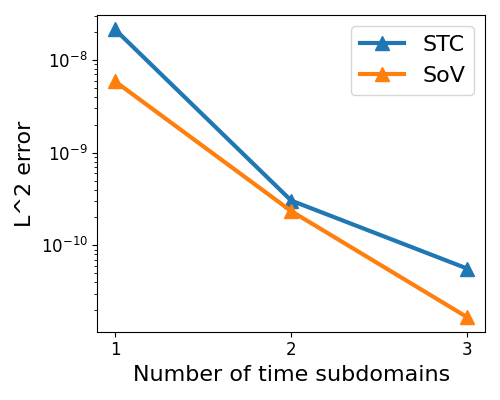}
	}
	\subfigure[Block time-marching v.s. ST-RFM.\label{exp:fig:sch1:ur:f}]{
		\includegraphics[width=.30\linewidth]{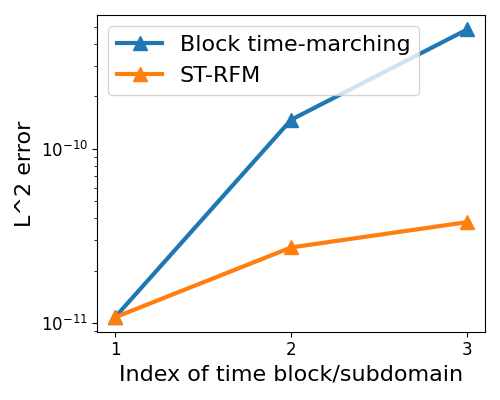}
	}
	\caption{Convergence behavior as different hyper-parameters are varied for Schr\"{o}dinger equation (real part). (a) the number of time subdomains for ST-RFM; (b) the number of time blocks in the block time-marching strategy; (c) the number of random features in each direction; (d) the number of collocation points in each direction; (e) comparison of STC and SoV; (f) comparison of the ST-RFM and the block time-marching strategy.}
	\label{exp:fig:sch1:ur}
\end{figure}

\begin{figure}[htbp]
	\centering
	\subfigure[Convergence w.r.t. $N_t$.\label{exp:fig:sch2:ui:a}]{
		\includegraphics[width=.30\linewidth]{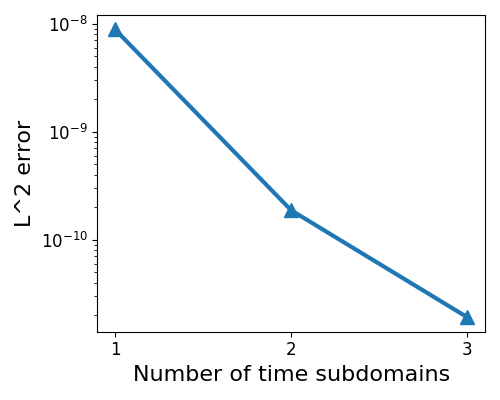}
	}
	\subfigure[Convergence w.r.t. $N_b$.\label{exp:fig:sch2:ui:b}]{
		\includegraphics[width=.30\linewidth]{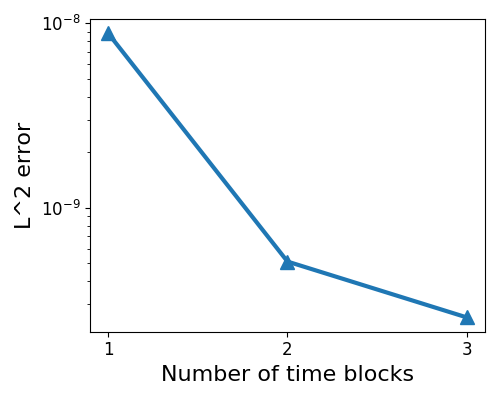}
	}
	\subfigure[Convergence w.r.t. $J_n$.\label{exp:fig:sch2:ui:c}]{
		\includegraphics[width=.30\linewidth]{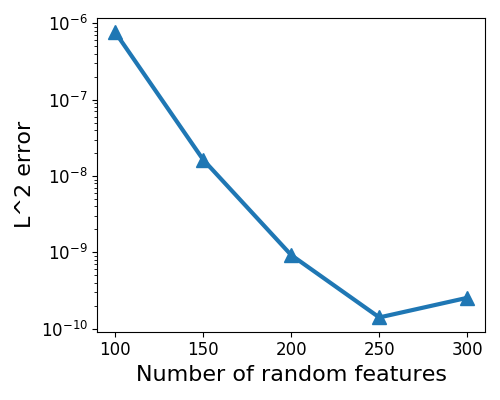}
	}
	\subfigure[Convergence w.r.t. $Q$.\label{exp:fig:sch2:ui:d}]{
		\includegraphics[width=.30\linewidth]{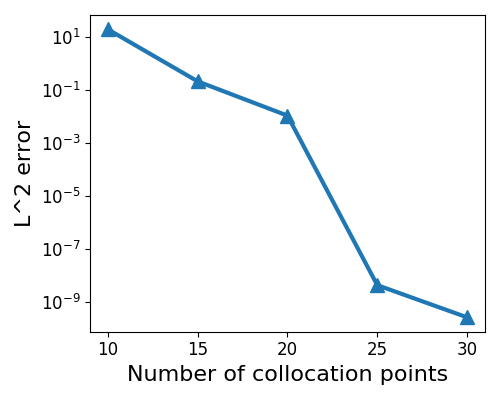}
	}
	\subfigure[STC v.s. SoV.\label{exp:fig:sch2:ui:e}]{
		\includegraphics[width=.30\linewidth]{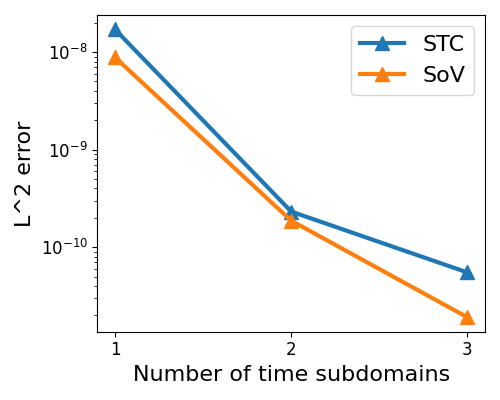}
	}
	\subfigure[Block time-marching v.s. ST-RFM.\label{exp:fig:sch2:ui:f}]{
		\includegraphics[width=.30\linewidth]{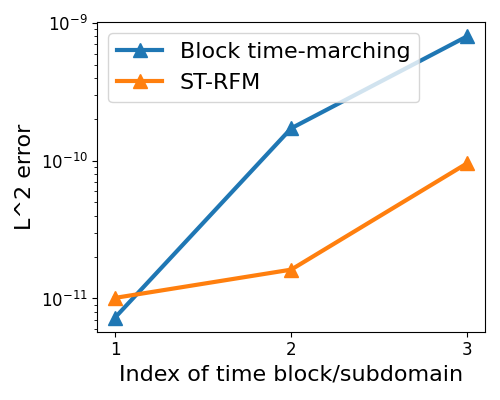}
	}
	\caption{Convergence behavior as different hyper-parameters are varied for Schr\"{o}dinger equation (imaginary part). (a) the number of time subdomains for ST-RFM; (b) the number of time blocks in the block time-marching strategy; (c) the number of random features in each direction; (d) the number of collocation points in each direction; (e) comparison of STC and SoV; (f) comparison of the ST-RFM and the block time-marching strategy.}
	\label{exp:fig:sch2:ui}
\end{figure}

\subsubsection{Verification of Theoretical Results} \label{sec:num:oned:verify}

First, we verify Assumption~\ref{assum:method:2} for all one-dimensional problems using the sufficient condition that the number of different eigenvalues of $\mathbf{B}$, denoted by \#$\mathrm{unique}\; \lambda_k$, equals to the number of random features $J_n$. Results are recorded in Table~\ref{exp:tab:Beigen}. The number of unique eigenvalues of $\mathbf{B}$ equals to $J_n$ for all one-dimensional problems and Assumption~\ref{assum:method:2} is verified.
\begin{table}[htbp]
	\centering
	\begin{tabular}{cc|cc|cc}
		\hline
		\multicolumn{2}{c|}{Heat} & \multicolumn{2}{c|}{Wave} & \multicolumn{2}{c}{Schr\"{o}dinger} \\
		\hline
		$J_n$ & \#unique $\lambda_k$ & $J_n$ & \#unique $\lambda_k$ & $J_n$ & \#unique $\lambda_k$ \\
		\hline
		100 & 100 & 100 & 100 & 200 & 200 \\
		150 & 150 & 150 & 150 & 300 & 300 \\
		200 & 200 & 200 & 200 & 400 & 400 \\
		250 & 250 & 250 & 250 & 500 & 500 \\
		300 & 300 & 300 & 300 & 600 & 600 \\
		350 & 350 & 350 & 350 & 700 & 700 \\
		400 & 400 & 400 & 400 & 800 & 800 \\
		\hline
	\end{tabular}
	\caption{Eigenvalue distribution of $\mathbf{B}$ in one-dimensional problems for the verification of Assumption~\ref{assum:method:2}.}
	\label{exp:tab:Beigen}
\end{table}

Next, we use the wave equation to verify the error estimate in Theorem~\ref{coro:method:3}. For the block time-marching strategy, we set $N_b=20, N_x=2, N_t=1, J_n=100, Q_{x}=10, Q_{t}=10$, and for ST-RFM, we set $N_b=1, N_x=2, N_t=20, J_n=100, Q_x=10, Q_t=10$. Eigenvalues of $\mathbf{B}$ are plotted in Figure~\ref{fig:exp:conv:a}, and the $L^2$ error is plotted in Figure~\ref{fig:exp:conv:b}. 

The largest modulus of eigenvalues of $\mathbf{B}$ from Figure~\ref{fig:exp:conv:a} is $187.75$, indicating that $\beta(x) \approx 13.70^{x}$, while the $L^2$ error in the block time-marching strategy increases exponentially  with the rate $17.70$ from Figure~\ref{fig:exp:conv:b}. Therefore, the lower bound estimate in Theorem~\ref{thm:method:4} is verified. In addition, the $L^2$ error in the ST-RFM remains almost flat, which is mainly constrolled by the approximation power of space-time random features.
\begin{figure}[htbp]
	\centering
	\subfigure[Modulus of $\lambda_k$.\label{fig:exp:conv:a}]{
		\includegraphics[width=.45\linewidth]{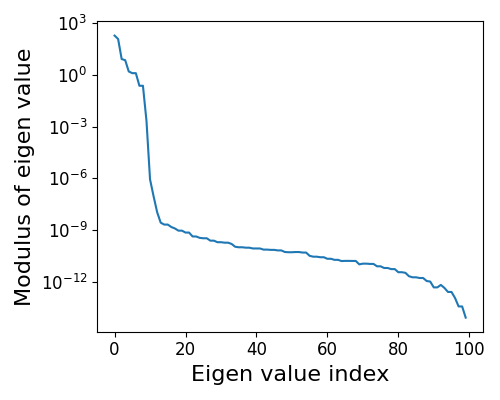}
	}
	\subfigure[$L^2$ error.\label{fig:exp:conv:b}]{
		\includegraphics[width=.45\linewidth]{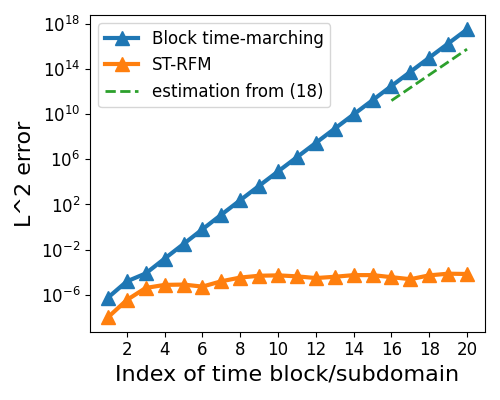}
	}
	\caption{(a) Eigenvalue distribution for the matrix $\mathbf{B}$; (b) $L^2$ error for the wave equation. We see a good agreement between Theorem \ref{thm:method:4},  Theorem \ref{thm:method:5}, and the actual numerical results.}
	\label{fig:exp:conv}
\end{figure}

\subsection{Two-dimensional problems} \label{sec:num:twod}

\subsubsection{Membrane Vibration over a Simple Geometry} \label{sec:num:twod:simple}

Consider the following problem
\begin{equation}
\left\{
\begin{aligned}
& \partial_{tt} u(x, y, t) - \alpha^2 \Delta u(x, y, t) = 0, & \quad (x, y), t \in \Omega \times [0, T], \\
& u(x, y, 0) = \phi(x, y), & \quad (x, y) \in \Omega, \\
& \partial_t u(x, y, 0) = \psi(x, y), & \quad (x, y) \in \Omega, \\
& u(x, y, t) = 0, & \quad (x, y) \in \partial \Omega \times [0, T], \\
\end{aligned}
\right.
\label{eqn:membrane}
\end{equation}
where $\Omega = [0, 5] \times [0, 4]$, $\alpha = 1$ and $T = 10$. The exact solution is chosen to be
\begin{equation}
u_e(x, y, t) = \sin(\mu x) \sin(\nu y) (2 \cos(\lambda t) + \sin(\lambda t)), \qquad \mu = \frac{2 \pi}{x_1 - x_0}, \nu = \frac{2 \pi}{y_1 - y_0}, \lambda = \sqrt{\mu^2 + \nu ^ 2},
\label{eqn:membrane:sol}
\end{equation}
and $\phi(x, y)$ and $\psi(x, y)$ are chosen accordingly.

Set the default hyper-parameters $N_x=N_y=2$, $N_t=5$, $Q_x=Q_y=Q_t=30$, $J_n=400$ and $N_b=1$. We report the convergence behavior with respect to different parameters in Figure~\ref{fig:exp:mem:conv}(a-c). In Figure~\ref{fig:exp:mem:conv:a}, we set $N_b=1$, $N_x=N_y=2$, $J_n=400$, $Q_x=Q_y=Q_t=30$ and $N_t=1,\cdots,6$ to verify the convergence with respect to $N_t$. In Figure~\ref{fig:exp:mem:conv:b}, we set $N_b=1$, $N_x=N_y=N_t=2$, $Q_x=Q_y=Q_t=30$ and $J_n=100,150,200,250,300,350,400$ to verify the convergence with respect to $J_n$. In Figure~\ref{fig:exp:mem:conv:c}, we set $N_b=5$, $N_x=N_y=N_t=5$, $J_n=400$ and $Q_x=Q_y=Q_t=10,15,20,25,30$ to verify the convergence with respect to the number of collocation points. A clear trend of spectral accuracy is observed for the ST-RFM in both spatial and temporal directions. Now, we compare STC and SoV in Figure~\ref{fig:exp:mem:conv:d}, when $N_b = 1$, $N_x = N_y = 2$, $N_t = 1, \cdots, 6$, $Q_x = Q_y = Q_t = 30$ and $J_n = 400$. Performances of STC and SoV are close.
\begin{figure}[htbp]
	\centering
	\subfigure[Convergence w.r.t. $N_t$. \label{fig:exp:mem:conv:a}]{
		\includegraphics[width=.23\linewidth]{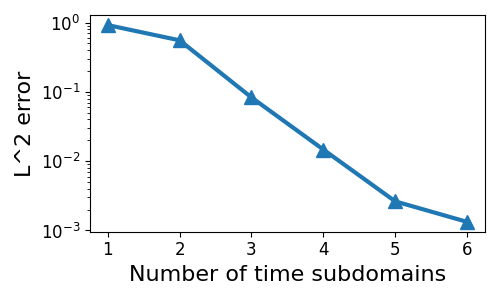}
	}
	\subfigure[Convergence w.r.t. $J_n$. \label{fig:exp:mem:conv:b}]{
		\includegraphics[width=.23\linewidth]{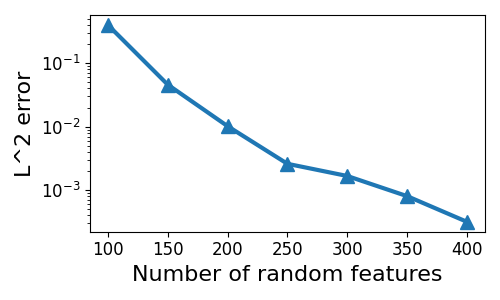}
	}
	\subfigure[Convergence w.r.t. $Q$. \label{fig:exp:mem:conv:c}]{
		\includegraphics[width=.23\linewidth]{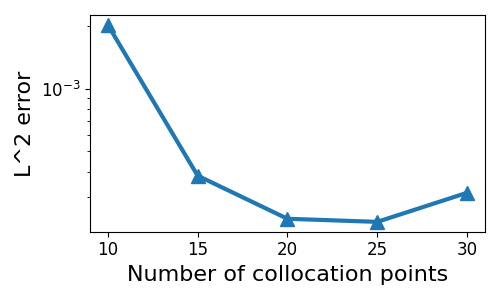}
	}
	\subfigure[STC v.s. SoV. \label{fig:exp:mem:conv:d}]{
		\includegraphics[width=.23\linewidth]{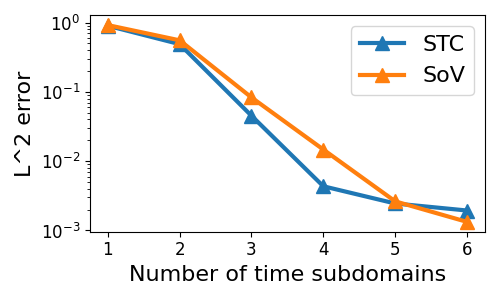}
	}
	\caption{Convergence behavior as different hyper-parameters are varied for the membrane vibration equation~\eqref{eqn:membrane}.} \label{fig:exp:mem:conv}
\end{figure}

\subsubsection{Membrane vibration over a Complex Geometry} \label{sec:num:twod:complex}

Consider a complex geometry $\Omega$ in Figure~\ref{exp:fig:geo}
\begin{figure}[ht]
	\centering
	\includegraphics[width=.5\linewidth]{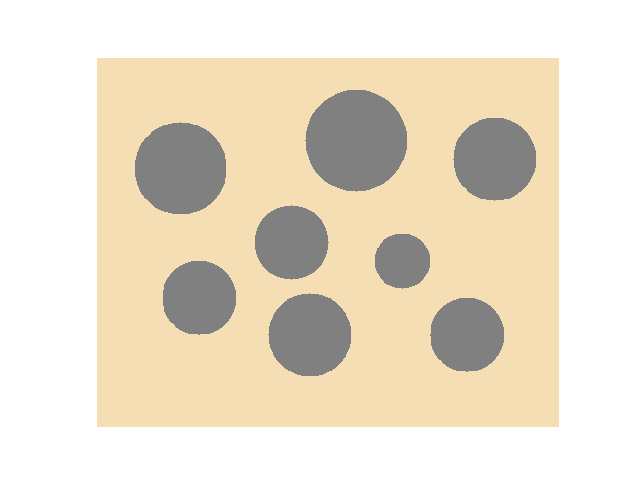}
	\caption{The complex geometry for the membrane vibration problem.}
	\label{exp:fig:geo}
\end{figure}
and the following the membrane vibration problem
\begin{equation}
\left\{
\begin{aligned}
& \partial_{tt} u(x, y, t) - \alpha^2 \Delta u(x, y, t) = 1, & \quad (x, y), t \in \Omega \times [0, T], \\
& u(x, y, 0) = \phi(x, y), & \quad (x, y) \in \Omega, \\
& \partial_t u(x, y, 0) = \psi(x, y), & \quad (x, y) \in \Omega, \\
& u(x, y, t) = 0, & \quad (x, y) \in \partial \Omega \times [0, T], \\
\end{aligned}
\right.
\label{eqn:membrane-cg}
\end{equation}
where $T=10.0$, $\alpha = 1$. The same $\phi(x, y)$ and $\psi(x, y)$ are used as in Section \ref{sec:num:twod:simple}. Set the default hyper-parameters $N_x=5$, $N_t=1$, $Q_x=30$, $Q_t=30$, $J_n=400$ and $N_b=5$. Numerical solutions of SoV at different times are plotted in Figure~\ref{exp:fig:membrane-cg}.
\begin{figure}[htbp]
	\centering
	\subfigure[Time=0.0.]{
		\includegraphics[width=.45\linewidth]{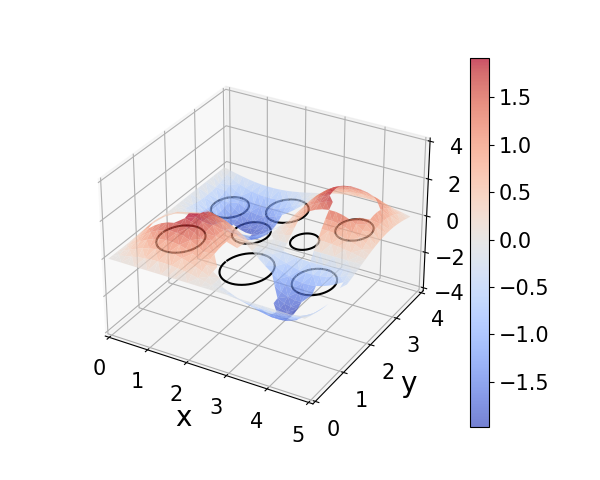}
	}
	\subfigure[Time=2.5.]{
		\includegraphics[width=.45\linewidth]{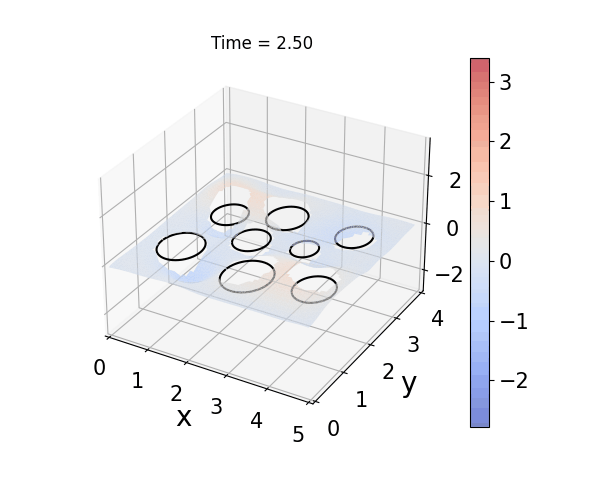}
	}
	\subfigure[Time=5.0.]{
		\includegraphics[width=.45\linewidth]{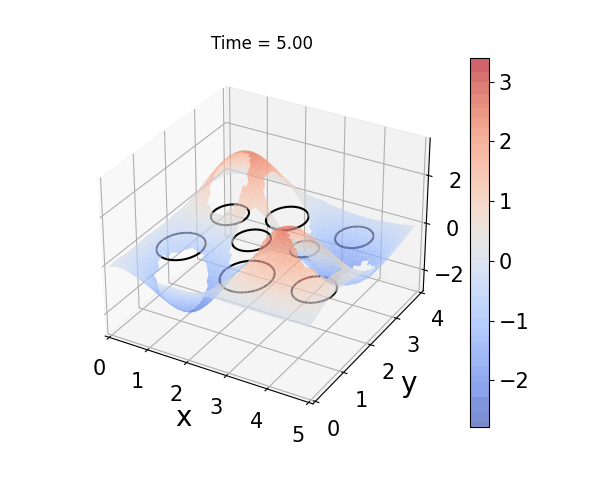}
	}
	\subfigure[Time=7.5.]{
		\includegraphics[width=.45\linewidth]{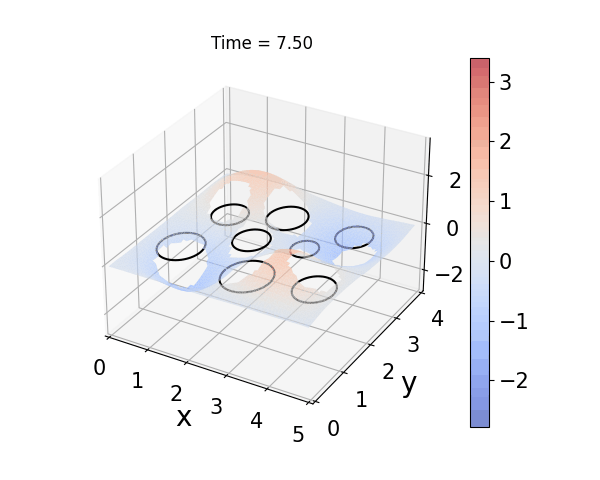}
	}
	\caption{Numerical solution of the membrane vibration problem over a complex geometry at different times.}
	\label{exp:fig:membrane-cg}
\end{figure}

No exact solution is available and numerical convergence is shown here. $L^2$ errors of solution $\| u_M - u_{ref} \|$ with respect to different parameters is recorded in Table~\ref{exp:tab:two1}, Table~\ref{exp:tab:two2} and Table~\ref{exp:tab:two3}. The solution with the largest degrees of freedom is chosen as the reference solution. As the parameter varies, the numerical solution converges to the reference solution, indicating the robustness of the ST-RFM in solving time-dependent partial differential equations over a complex geometry.
\begin{table}[htbp]
	\centering
	\begin{tabular}{cccccc}
		\hline
		\# Time Blocks & $N_x$/$N_y$ & $N_t$ & $J_n$ & $Q_x$/$Q_y$/$Q_t$ & Solution error \\
		\hline
		\hline
		1 & 2 & 1 & 400 & 30 & 5.98e-1 \\
		1 & 2 & 2 & 400 & 30 & 7.07e-2 \\
		1 & 2 & 3 & 400 & 30 & 4.57e-2 \\
		1 & 2 & 4 & 400 & 30 & 3.89e-2 \\
		1 & 2 & 5 & 400 & 30 & 3.31e-2 \\
		1 & 2 & 6 & 400 & 30 & Reference \\
		\hline
	\end{tabular}
	\caption{$L^2$ error with respect to the number of time sub-domains.}
	\label{exp:tab:two1}
\end{table}

\begin{table}[htbp]
	\centering
	\begin{tabular}{cccccc}
		\hline
		\# Time Blocks & $N_x$/$N_y$ & $N_t$ & $J_n$ & $Q_x$/$Q_y$/$Q_t$ & Solution error \\
		\hline
		\hline
		5 & 2 & 2 & 400 & 10 & 7.71e-2 \\
		5 & 2 & 2 & 400 & 15 & 2.16e-2 \\
		5 & 2 & 2 & 400 & 20 & 1.44e-2 \\
		5 & 2 & 2 & 400 & 25 & 1.25e-2 \\
		5 & 2 & 2 & 400 & 30 & Reference \\
		\hline
	\end{tabular}
	\caption{$L^2$ error with respect to the number of collocation points.}
	\label{exp:tab:two2}
\end{table}

\begin{table}[htbp]
	\centering
	\begin{tabular}{cccccc}
		\hline
		\# Time Blocks & $N_x$/$N_y$ & $N_t$ & $J_n$ & $Q_x$/$Q_y$/$Q_t$ & Solution error \\
		\hline
		\hline
		5 & 2 & 2 & 100 & 30 & 1.05e-1 \\
		5 & 2 & 2 & 150 & 30 & 5.25e-2 \\
		5 & 2 & 2 & 200 & 30 & 3.96e-2 \\
		5 & 2 & 2 & 250 & 30 & 2.82e-2 \\
		5 & 2 & 2 & 300 & 30 & 1.80e-2 \\
		5 & 2 & 2 & 350 & 30 & 1.46e-2 \\
		5 & 2 & 2 & 400 & 30 & Reference \\
		\hline
	\end{tabular}
	\caption{$L^2$ error with respect to the number of random features.}
	\label{exp:tab:two3}
\end{table}

\section{Concluding Remarks} \label{sec:con}
In this work, we study numerical algorithms for solving time-dependent partial differential equations in the framework of the random feature method.	
Two types of random feature functions are considered: space-time concatenation random feature functions (STC) and space-time separation-of-variables random feature functions (SoV). A space-time partition of unity is used to piece together local random feature functions to approximate the solution. 	
We tested these ideas for a number of time-dependent problems. Our numerical results show that ST-RFM with both STC and SoV has  spectral accuracy in	
space and time. The error in ST-RFM remains almost flat as the number of time subdomains increases, while the error grows exponentially fast when the block time-marching strategy is used. Consistent theoretical error estimates are also proved. A two-dimensional problem over a complex geometry is used to show that the method is insensitive to the complexity of the underlying domain.

\section*{Acknowledgments}
The work is supported by National Key R\&D Program of China (No. 2022YFA1005200 and No. 2022YFA1005203), NSFC Major Research Plan -  Interpretable and General-purpose Next-generation Artificial Intelligence (No. 92270001 and No. 92270205), Anhui Center for Applied Mathematics, and the Major Project of Science \& Technology of Anhui Province (No. 202203a05020050).

\bibliographystyle{amsplain}

\providecommand{\bysame}{\leavevmode\hbox to3em{\hrulefill}\thinspace}
\providecommand{\MR}{\relax\ifhmode\unskip\space\fi MR }
\providecommand{\MRhref}[2]{%
  \href{http://www.ams.org/mathscinet-getitem?mr=#1}{#2}
}
\providecommand{\href}[2]{#2}

\end{document}